\documentclass{amsart}
\usepackage[utf8]{inputenc}
\usepackage{amssymb}
\usepackage{hyperref}
\usepackage[final]{showkeys}

\input xy
\xyoption{all}

\theoremstyle{definition}
\newtheorem{mydef}{Definition}[section]
\newtheorem{lem}[mydef]{Lemma}
\newtheorem{thm}[mydef]{Theorem}
\newtheorem{cor}[mydef]{Corollary}

\newtheorem{question}[mydef]{Question}
\newtheorem{hypothesis}[mydef]{Hypothesis}

\newtheorem{defin}[mydef]{Definition}

\newtheorem{remark}[mydef]{Remark}
\newtheorem{notation}[mydef]{Notation}
\newtheorem{fact}[mydef]{Fact}

\newcommand{\fct}[2]{{}^{#1}#2}

\newcommand{\ba}{\bar{a}}
\newcommand{\bb}{\bar{b}}
\newcommand{\bc}{\bar{c}}
\newcommand{\bd}{\bar{d}}

\newcommand{\sea}{\mathfrak{C}}

\newcommand{\cf}[1]{\text{cf} (#1)}
\newcommand{\seq}[1]{\langle #1 \rangle}
\newcommand{\rest}{\upharpoonright}

\newcommand{\bmu}{\bar \mu}

\newcommand{\is}{\mathfrak{i}}

\newcommand{\isssp}[1]{\is_{#1\text{-strong-spl}}}
\newcommand{\issp}[1]{\is_{#1\text{-spl}}}
\newcommand{\isdiv}{\is_{\text{div}}}
\newcommand{\isf}[1]{\is_{#1\text{-forking}}}

\newcommand{\K}{\mathbf{K}}

\def\lta{<_{\K}}
\def\lea{\le_{\K}}

\def\ltu{\lta^{\text{univ}}}

\newbox\noforkbox \newdimen\forklinewidth
\forklinewidth=0.3pt \setbox0\hbox{$\textstyle\smile$}
\setbox1\hbox to \wd0{\hfil\vrule width \forklinewidth depth-2pt
 height 10pt \hfil}
\wd1=0 cm \setbox\noforkbox\hbox{\lower 2pt\box1\lower
2pt\box0\relax}
\def\unionstick{\mathop{\copy\noforkbox}\limits}

\setbox0\hbox{$\textstyle\smile$}
\setbox1\hbox to \wd0{\hfil{\sl /\/}\hfil} \setbox2\hbox to
\wd0{\hfil\vrule height 10pt depth -2pt width
               \forklinewidth\hfil}
\newbox\doesforkbox
\setbox\doesforkbox\hbox{\lower 0pt\box1 \lower
2pt\box2\lower2pt\box0\relax}

\newcommand{\nf}{\unionstick}

\newcommand{\nfs}[4]{#2 \nf_{#1}^{#4} #3}

\def\1nf{\unionstick^{(1)}}
\def\2nf{\unionstick^{(2)}}

\newcommand{\gtp}{\text{gtp}}

\newcommand{\gS}{\text{gS}}

\newcommand{\Ll}{\mathbb{L}}

\newcommand{\EM}{\operatorname{EM}}

\newcommand{\BI}{\mathbf{I}}
\newcommand{\BJ}{\mathbf{J}}
\newcommand{\slc}{\bmu}

\newcommand{\clc}{\kappa}
\newcommand{\clcwk}{\kappa^{\text{wk}}}
\newcommand{\clcc}{\kappa^{\text{cont}}}

\newcommand{\bclp}[1]{\underline{\kappa}^{\text{#1}}}
\newcommand{\bclc}{\bclp{}}
\newcommand{\bclcwk}{\bclp{wk}}
\newcommand{\bclcc}{\bclp{cont}}

\newcommand{\uchi}{\underline{\chi}}

\newcommand{\hanf}[1]{h (#1)}

\newcommand{\LS}{\text{LS}}

\newcommand{\plus}[1]{{#1}_r}
\newcommand{\kappap}{\plus{\kappa}}

\newcommand{\Kmhp}[1]{K^{#1\text{-mh}}}
\newcommand{\Kmh}[1]{\Kmhp{\lambda}}

\newcommand{\SE}{\text{SE}}
\newcommand{\Emin}[1]{E_{\text{min},#1, \alpha}}

\newcommand{\REG}{\operatorname{REG}}
\newcommand{\Stab}{\operatorname{Stab}}
\newcommand{\Mod}{\operatorname{Mod}}

\title[Toward a stability theory of tame AECs]{Toward a stability theory of tame abstract elementary classes}
\author{Sebastien Vasey}
\email{sebv@math.harvard.edu}
\urladdr{http://math.harvard.edu/\textasciitilde sebv/}
\address{Department of Mathematical Sciences, Carnegie Mellon University, Pittsburgh, Pennsylvania, USA}
\address{Current address: Department of Mathematics \\ Harvard University \\ Cambridge, Massachusetts, USA}

\date{\today } % delete this line to display the current date

\begin{document}

\begin{abstract}
  We initiate a systematic investigation of the abstract elementary classes that have amalgamation, satisfy tameness (a locality property for orbital types), and are stable (in terms of the number of orbital types) in some cardinal. Assuming the singular cardinal hypothesis (SCH), we prove a full characterization of the (high-enough) stability cardinals, and connect the stability spectrum with the behavior of saturated models.

  We deduce (in ZFC) that if a class is stable on a tail of cardinals, then it has no long splitting chains (the converse is known). This indicates that there is a clear notion of superstability in this framework.

  We also present an application to homogeneous model theory: for $D$ a homogeneous diagram in a first-order theory $T$, if $D$ is both stable in $|T|$ and categorical in $|T|$ then $D$ is stable in all $\lambda \ge |T|$.
\end{abstract}

\maketitle

\tableofcontents

\section{Introduction}

\subsection{Motivation and history}

Abstract elementary classes (AECs) are partially ordered classes $\K = (K, \lea)$ which satisfy several of the basic category-theoretic properties of classes of the form $(\Mod (T), \preceq)$ for $T$ a first-order theory. They were introduced by Saharon Shelah in the late seventies \cite{sh88} and encompass infinitary logics such as $\Ll_{\lambda^+, \omega} (Q)$ as well as several algebraic examples. One of Shelah's test questions is the eventual categoricity conjecture: an AEC categorical in \emph{some} high-enough cardinal should be categorical in \emph{all} high-enough cardinals.

Toward an approximation, work of Makkai and Shelah \cite{makkaishelah} studied classes of models of an $\Ll_{\kappa, \omega}$ theory categorical in a high-enough cardinal, where $\kappa$ is strongly compact. They proved \cite[1.13]{makkaishelah} that such a class has (eventual) amalgamation, joint embedding, and no maximal models. Thus one can work inside a monster model and look at the corresponding orbital types. Makkai and Shelah established that the orbital types correspond to certain syntactic types, implying in particular that two orbital types are equal if all their restrictions of size less than $\kappa$ are equal. They then went on to develop some superstability theory and concluded that categoricity in some high-enough successor implies categoricity in all high-enough cardinals.

A common theme of recent work on AECs is to try to replace large cardinal hypotheses with their model-theoretic consequences. For example, regardless of whether there are large cardinals, many classes of interests have a monster model and satisfy a locality property for their orbital types (see the introduction to \cite{tamenessone} or the list of examples in the recent survey \cite{bv-survey-bfo}). Toward that end, Grossberg and VanDieren made the locality property isolated by Makkai and Shelah (and later also used by Shelah in another work \cite{sh394}) into a definition: Call an AEC \emph{$\mu$-tame} if its orbital types are determined by their $\mu$-sized restrictions. Will Boney \cite{tamelc-jsl} has generalized the first steps in the work of Makkai and Shelah to AECs, showing that tameness follows from a large cardinal axiom (amalgamation also follows if one assumes categoricity). Earlier, Shelah had shown that Makkai and Shelah's downward part of the transfer holds assuming amalgamation (but not tameness) \cite{sh394} and Grossberg and VanDieren used Shelah's proof (their initial motivation for isolating tameness) to show that the upward part of the transfer holds in tame AECs with amalgamation.

Recently, the superstability theory of tame AECs with a monster model has seen a lot of development \cite{ext-frame-jml, ss-tame-jsl, bv-sat-afml, vv-symmetry-transfer-afml, gv-superstability-jsl} and one can say that most of Makkai and Shelah's work has been generalized to the tame context (see also \cite[D.9(3)]{baldwinbook09}). New concepts not featured in the Makkai and Shelah paper, such as good frames and limit models, have also seen extensive studies (e.g.\ in the previously-cited papers and in Shelah's book \cite{shelahaecbook}). The theory of superstability for AECs has had several applications, including a full proof of Shelah's eventual categoricity conjecture in universal classes \cite{ap-universal-apal, categ-universal-2-selecta}.

While we showed with Grossberg \cite{gv-superstability-jsl} that several possible definitions of superstability are all equivalent in the tame case, it was still open \cite[1.7]{gv-superstability-jsl} whether stability on a tail of cardinals implied these possible definitions (e.g.\ locality of forking).

The present paper answers positively (see Corollary \ref{ss-cor}) by developing the theory of \emph{strictly stable} tame AECs with a monster model. We emphasize that this is \emph{not} the first paper on strictly stable AECs. In their paper introducing tameness \cite{tamenessone}, Grossberg and VanDieren proved several fundamental results (see also \cite{b-k-vd-spectrum}). Shelah \cite[\S4,\S5]{sh394} has made some important contributions without even assuming tameness; see also his work on universal classes \cite[V.E]{shelahaecbook2}. Several recent works \cite{bg-apal,sv-infinitary-stability-afml,bv-sat-afml, limit-strictly-stable-v3} establish results on independence, the first stability cardinal, chains of saturated models, and limit models. The present paper aims to put these papers together and improve some of their results using either the superstability machinery mentioned above or (in the case of Shelah's tameness-free results) assuming tameness.

\subsection{Outline of the main results}

Fix an $\LS (\K)$-tame AEC $\K$ with a monster model. Assume that $\K$ is stable (defined by counting Galois types) in some cardinal. Let $\uchi (\K)$ be the class of regular cardinals $\chi$ such that for all high-enough stability cardinals $\mu$, any type over the union of a $(\mu, \chi)$-limit chain $\seq{M_i : i < \chi}$ does not $\mu$-split over some $M_i$. When $\K$ is an elementary class, $\uchi (\K)$ is an end-segment of cardinals whose minimum is $\kappap (T)$ (the least regular cardinal greater than or equal to $\kappa (T)$). Note that in general we do not know whether $\uchi (\K)$ must be an end segment of regular cardinals or whether it can have gaps (we can give a locality condition implying that it is an end segment, see Corollary \ref{locality-cor} and Theorem \ref{loc-implies-cont}).

Using results from the theory of averages in tame AECs (developed in \cite{bv-sat-afml, gv-superstability-jsl}), we show assuming the singular cardinal hypothesis (SCH\footnote{That is, for every infinite singular cardinal $\lambda$, $\lambda^{\cf{\lambda}} = 2^{\cf{\lambda}} + \lambda^+$.}) that for all high-enough cardinals $\mu$, $\K$ is stable in $\mu$ if and only if $\cf{\mu} \in \uchi (\K)$ (see Corollary \ref{non-zfc-stab-spec}). The right to left direction is implicit in the author's earlier work \cite[5.7]{ss-tame-jsl} but the left to right direction is new.

A consequence of the proof of Corollary \ref{non-zfc-stab-spec} is that stability on a tail implies that $\uchi (\K)$ contains all regular cardinals (Corollary \ref{ss-cor}; note that this is in ZFC). Thus we propose the following definition:

\textbf{Definition \ref{ss-chi-def}.} An $\LS (\K)$-tame AEC with a monster model is \emph{superstable} if $\uchi (\K)$ contains all regular cardinals.

Combining \cite{gv-superstability-jsl} with the results of the present paper, we obtain that Definition \ref{ss-chi-def} is equivalent to many other potential candidate definitions of superstability (such as stability on a tail, or the union of a chain of $\lambda$-saturated models to be $\lambda$-saturated). Thus we believe that \emph{in the setup of tame AECs with a monster model}, this is the right definition.

The present paper also shows that $\uchi (\K)$ connects the stability spectrum with the behavior of saturated models: assuming SCH, a stable tame AEC with a monster model has a saturated model in a high-enough $\lambda$ if and only if [$\lambda = \lambda^{<\lambda}$ or $\K$ is stable in $\lambda$]. In ZFC, we deduce that having saturated models on a tail of cardinals implies superstability (Corollary \ref{ss-sat}). We conclude with Theorem \ref{charact-thm}, giving (in ZFC) several equivalent definitions of $\uchi (\K)$, in terms of uniqueness of limit models, existence of saturated models, or the stability spectrum. Sections \ref{indisc-seq}-\ref{cont-sec-thy} adapt the study of strict stability from \cite{sh394} to the tame context and use a weak continuity property for splitting (as assumed in \cite{limit-strictly-stable-v3}) to improve on some of the results mentioned earlier. Note that while Section \ref{cont-sec-thy} contains the main results, it depends on Sections \ref{indisc-seq}-\ref{strong-split-sec} (which are more technical). Section \ref{applications-sec} gives a quick application to homogeneous model theory: categoricity in $|T|$ and stability in $|T|$ imply stability in all $\lambda \ge |T|$.

The reader may ask how SCH is used in the above results. Roughly, it makes cardinal arithmetic well-behaved enough that for any big-enough cardinal $\lambda$, $\K$ will either be stable in $\lambda$ or in unboundedly many cardinals below $\lambda$. This is connected to defining the locality cardinals in $\uchi (\K)$ using chains rather than as the least cardinal $\kappa$ for which every type does not fork over a set of size less than $\kappa$ (indeed, in AECs it is not even clear what exact form such a definition should take). Still several results of this paper hold (in ZFC) for ``most'' cardinals, and the role of SCH is only to deduce that ``most'' means ``all''.

By a result of Solovay \cite{solovay-sch}, SCH holds above a strongly compact. Thus our results which assume SCH hold also above a strongly compact. This shows that a stability theory (not just a superstability theory) can be developed in the context of the Makkai and Shelah paper, partially answering \cite[6.15]{sh702}.

\subsection{Future work}

We believe that an important test question is whether the aforementioned SCH hypothesis can be removed:

\begin{question}\label{stab-spec-q}
  Let $\K$ be an $\LS (\K)$-tame AEC with a monster model. Can one characterize the stability spectrum in ZFC?
\end{question}

By the present work, the answer to Question \ref{stab-spec-q} is positive assuming the existence of large cardinals.

Apart from $\uchi (\K)$, several other cardinal parameters ($\lambda (\K)$, $\lambda' (\K)$, and $\slc (\K)$) are defined in this paper. Under some assumptions, we can give loose bounds on these cardinals (see e.g.\ Theorem \ref{locality-card-bounds}) but focus on eventual behavior. We believe it is a worthy endeavor (analog to the study of the behavior of the stability spectrum below $2^{|T|}$ in first-order) to try to say something more on these cardinals.

\subsection{Notes}
The background required to read this paper is a solid knowledge of tame AECs (as presented for example in \cite{baldwinbook09}). Familiarity with \cite{ss-tame-jsl} would be very helpful. Results from the recent literature which we rely on can be used as black boxes.

This paper was written while the author was working on a Ph.D.\ thesis under the direction of Rami Grossberg at Carnegie Mellon University and he would like to thank Professor Grossberg for his guidance and assistance in his research in general and in this work specifically. The author also thanks John T.\ Baldwin and the referee for valuable comments that helped improve the presentation of the paper.

Note that at the beginning of several sections, we make global hypotheses assumed throughout the section. In the statement of the main results, these global hypotheses will be repeated.

\section{Preliminaries}

\subsection{Basic notation}

\begin{notation}\label{basic-notation} \
  \begin{enumerate}
  \item We use the Hanf number notation from \cite[4.24]{baldwinbook09}: for $\lambda$ an infinite cardinal, write $\hanf{\lambda} := \beth_{\left(2^{\lambda}\right)^+}$. When $\K$ is an AEC clear from context, $H_1 := \hanf{\LS (\K)}$.
  \item Let $\REG$ denote the class of regular cardinals.
  \end{enumerate}
\end{notation}

\subsection{Monster model, Galois types, and tameness}\label{monster-subsec}

We say that an AEC $\K$ \emph{has a monster model} if it has amalgamation, joint embedding, and arbitrarily large models. Equivalently, it has a (proper class sized) model-homogeneous universal model $\sea$. When $\K$ has a monster model, we fix such a $\sea$ and work inside it. Note that for our purpose amalgamation is the only essential property. Once we have it, we can partition $\K$ into disjoint pieces, each of which has joint embedding (see for example \cite[16.14]{baldwinbook09}). Further, for studying the eventual behavior of $\K$ assuming the existence of arbitrarily large models is natural.

We use the notation of \cite{sv-infinitary-stability-afml} for Galois types. In particular, $\gtp (\bb / A; N)$ denotes the Galois type of the sequence $\bb$ over the set $A$, as computed in $N \in \K$. In case $\K$ has a monster model $\sea$, we write $\gtp (\bb / A)$ instead of $\gtp (\bb / A; \sea)$. In this case, $\gtp (\bb / A) = \gtp (\bc / A)$ if and only if there exists an automorphism $f$ of $\sea$ fixing $A$ such that $f (\bb) = \bc$.

Observe that the definition of Galois types is completely semantic. Tameness is a locality property for types isolated by Grossberg and VanDieren \cite{tamenessone} that, when it holds, allows us to recover some of the syntactic properties of first-order types. For a cardinal $\mu \ge \LS (\K)$, we say that an AEC $\K$ with a monster model is \emph{$\mu$-tame} if whenever $\gtp (b / M) \neq \gtp (c / M)$, there exists $M_0 \in \K_{\le \mu}$ such that $M_0 \lea M$ and $\gtp (b / M_0) \neq \gtp (c / M_0)$. When assuming tameness in this paper, we will usually assume that $\K$ is $\LS (\K)$-tame. Indeed if $\K$ is $\mu$-tame we can just replace $\K$ by $\K_{\ge \mu}$. Then $\LS (\K_{\ge \mu}) = \mu$, so $\K_{\ge \mu}$ will be $\LS (\K_{\ge \mu})$-tame.

Concepts such as stability and saturation are defined as in the first-order case but using Galois type. For example, an AEC $\K$ with a monster model is \emph{stable in $\mu$} if $|\gS (M)| \le \mu$ for every $M \in \K_\mu$. For $\mu > \LS (\K)$, a model $M \in \K$ is \emph{$\mu$-saturated} if every Galois type over a $\lea$-substructure of $M$ of size less than $\mu$ is realized in $M$. In the literature, these are often called ``Galois stable'' and ``Galois saturated'', but we omit the ``Galois'' prefix since there is no risk of confusion in this paper.

As in \cite[4.3]{sh394}:

\begin{defin}\label{op-def}
  Let $\K$ be an AEC with a monster model and let $\alpha$ be a non-zero cardinal. We say that \emph{$\K$ has the $\alpha$-order property} if for every $\theta$ there exists $\seq{\ba_i : i < \theta}$ such that for all $i < \theta$, $\ell (\ba_i) = \alpha$ and for all $i_0 < i_1 < \theta$, $j_0 < j_1 < \theta$, $\gtp (\ba_{i_0} \ba_{i_1} / \emptyset) \neq \gtp (\ba_{j_1} \ba_{j_0} / \emptyset)$.
\end{defin}

\subsection{Independence relations}

An \emph{abstract class (AC)} is a partial order $\K = (K, \lea)$ where $K$ is a class of structures in a fixed vocabulary $\tau (\K)$, $\K$ is closed under isomorphisms, and $M \lea N$ implies $M \subseteq N$ (the definition is due to Rami Grossberg \cite{grossbergbook}).

In this paper, an \emph{independence relation} will be a pair $(\K, \nf)$, where:

\begin{enumerate}
\item $\K$ is a coherent\footnote{that is, whenever $M_0 \subseteq M_1 \lea M_2$ and $M_0 \lea M_2$, we have that $M_0 \lea M_1$.} abstract class with amalgamation.
\item $\nf$ is a $4$-ary relation so that:
  \begin{enumerate}
  \item $\nf (M, A, B, N)$ implies $M \lea N$, $A, B \subseteq |N|$, $|A| \le 1$. We write $\nfs{M}{A}{B}{N}$.
  \item $\nf$ satisfies invariance, normality, and monotonicity (see \cite[3.6]{indep-aec-apal} for the definitions).
  \item We also ask that $\nf$ satisfies base monotonicity: if $\nfs{M_0}{A}{B}{N}$, $M_0 \lea M \lea N$, and $|M| \subseteq B$, then $\nfs{M}{A}{B}{N}$.
  \end{enumerate}
\end{enumerate}

Note that this definition differs slightly from that in \cite[3.6]{indep-aec-apal}: there additional parameters are added controlling the size of the left and right hand side, and base monotonicity is not assumed. Here, the size of the left hand side is at most 1 and the size of the right hand side is not bounded. So in the terminology of \cite{indep-aec-apal}, we are defining a $(\le 1, [0, \infty))$-independence relation with base monotonicity.

  When $\is = (\K, \nf)$ is an independence relation and $p \in \gS (B; N)$ (we make use of Galois types over sets, see \cite[2.16]{sv-infinitary-stability-afml}), we say that $p$ \emph{does not $\is$-fork over $M$} if $\nfs{M}{a}{B}{N}$ for some (any) $a$ realizing $p$ in $N$. When $\is$ is clear from context, we omit it and just say that $p$ \emph{does not fork over $M$}.

  The following independence notion is central. It was introduced by Shelah in \cite[3.2]{sh394}.

  \begin{defin}\label{splitting-def}
    Let $\K$ be a coherent abstract class with amalgamation, let $M \lea N$, $p \in \gS (N)$, and let $\mu \ge \|M\|$. We say that $p$ \emph{$\mu$-splits over $M$} if there exists $N_1, N_2 \in \K_{\le \mu}$ and $f$ such that $M \lea N_\ell \lea N$ for $\ell = 1,2$, $f: N_1 \cong_{M} N_2$, and $f (p \rest N_1) \neq p \rest N_2$.

    For $\lambda$ an infinite cardinal, we write $\issp{\mu} (\K_\lambda)$ for the independence relation with underlying class $\K_\lambda$ and underlying independence notion non $\mu$-splitting.
  \end{defin}

  \subsection{Universal orderings and limit models}\label{univ-sec}

  Work inside an abstract class $\K$. For $M \lta N$, we say that $N$ is \emph{universal over $M$} (and write $M \ltu N$) if for any $M' \in \K$ with $M \lea M'$ and $\|M'\| = \|M\|$, there exists $f: M' \xrightarrow[M]{} N$. For a cardinal $\mu$ and a limit ordinal $\delta < \mu^+$, we say that $N$ is \emph{$(\mu, \delta)$-limit over $M$} if there exists an increasing continuous chain $\seq{N_i : i \le \delta}$ such that $N_0 = M$, $N_\delta = N$, and for any $i < \delta$, $N_i$ is in $\K_\mu$ and $N_{i + 1}$ is universal over $N_i$. For $A \subseteq \mu^+$ a set of limit ordinals, we say that $N$ is \emph{$(\mu, A)$-limit over $M$} if there exists $\gamma \in A$ such that $N$ is $(\mu, \gamma)$-limit over $M$. $(\mu, \ge \delta)$-limit means $(\mu, [\delta, \mu^+) \cap \REG)$-limit. We will use without mention the basic facts about limit models in AECs: existence (assuming stability and a monster model) and uniqueness when they have the same cofinality. See \cite{gvv-mlq} for an introduction to the theory of limit models.

    \subsection{Locality cardinals for independence}

    One of the main object of study of this paper is $\uchi (\K)$ (see Definition \ref{chi-def}), which roughly is the class of regular cardinals $\chi$ such that for any increasing continuous chain $\seq{M_i : i \le \chi}$ where each model is universal over the previous one and for any $p \in \gS (M_\chi)$ there exists $i < \chi$ such that $p$ does not $\|M_i\|$-split over $M_i$. Interestingly, we cannot rule out the possibility that there are gaps in $\uchi (\K)$, i.e.\ although we do not have any examples, it is conceivable that there are regular $\chi_0 < \chi_1 < \chi_2$ such that chains of length $\chi_0$ and $\chi_2$ have the good property above but chains of length $\chi_1$ do not). This is why we follow Shelah's approach from \cite{sh394} (see in particular the remark on top of p.~275 there) and define \emph{classes} of locality cardinals, rather than directly taking a minimum (as in for example \cite[4.3]{tamenessone}). We give a sufficient locality condition implying that there are no gaps in $\uchi (\K)$ (see Theorem \ref{loc-implies-cont}).

    The cardinals $\bclcwk$ are already in \cite[4.8]{sh394}, while $\bclcc$ is used in the proof of the Shelah-Villaveces theorem \cite[2.2.1]{shvi635}, see also \cite{shvi-notes-apal}.

    \begin{defin}[Locality cardinals]\label{loc-def}
  Let $\is$ be an independence relation. Let $R$ be a partial order on $\K$ extending $\lea$.

    \begin{enumerate}
    \item $\bclc (\is, R)$ is the set of regular cardinals $\chi$ such that whenever $\seq{M_i : i < \chi}$ is an $R$-increasing chain\footnote{that is, $M_i R M_{j}$ for all $i < j < \chi$.}, $N \in \K$ is such that $M_i \lea N$ for all $i < \chi$, and $p \in \gS (\bigcup_{i < \chi} |M_i|; N)$, there exists $i < \chi$ such that $p$ does not fork over $M_i$.
    \item $\bclcwk (\is, R)$ is the set of regular cardinals $\chi$ such that whenever $\seq{M_i : i < \chi}$ is an $R$-increasing chain, $N \in \K$ is such that $M_i \lea N$ for all $i < \chi$, and $p \in \gS (\bigcup_{i < \chi} |M_i|; N)$, there exists $i < \chi$ such that $p \rest M_{i + 1}$ does not fork over $M_i$.
    \item $\bclcc (\is, R)$ is the set of regular cardinals $\chi$ such that whenever $\seq{M_i : i < \chi}$ is an $R$-increasing chain, $N \in \K$ is such that $M_i \lea N$ for all $i < \chi$, and $p \in \gS (\bigcup_{i < \chi} |M_i|; N)$, if $p \rest M_i$ does not fork over $M_0$ for all $i < \chi$, then $p$ does not fork over $M_0$.
    \end{enumerate}
    \end{defin}

    When $R$ is $\lea$, we omit it. In this paper, $R$ will mostly be $\ltu$ (see Section \ref{univ-sec}).

    \begin{remark}
      The behavior at singular cardinals has some interests (see for example \cite[11(4)]{shvi-notes-apal}), but we focus on regular cardinals in this paper.
    \end{remark}

    Note that $\bclcwk (\is, R)$ is an end segment of regular cardinals (so it has no gaps): if $\chi_0 < \chi_1$ are regular cardinals and $\chi_0 \in \bclcwk (\is, R)$, then $\chi_1 \in \bclcwk (\is, R)$. In section \ref{cont-sec} we will give conditions under which $\bclc (\is, R)$ and $\bclcc (\is, R)$ are also end segments. In this case, the following cardinals are especially interesting (note the absence of line under $\kappa$):

\begin{defin}
  $\clc (\is, R)$ is the least regular cardinal $\chi \in \bclc (\is, R)$ such that for any regular cardinals $\chi' > \chi$, we have that $\chi' \in \bclc (\is, R)$. Similarly define $\clcwk (\is, R)$ and $\clcc (\is, R)$.
\end{defin}

The following is given by the proof of \cite[11(1)]{shvi-notes-apal}:

\begin{fact}\label{locality-fact}
  Let $\is = (\K, \nf)$ be an independence relation. Let $R$ be a  partial order on $\K$ extending $\lea$.

  We have that $\bclcwk (\is, R) \cap \bclcc (\is, R) \subseteq \bclc (\is, R)$.
\end{fact}
\begin{cor}\label{locality-cor}
  Let $\is = (\K, \nf)$ be an independence relation. Let $R$ be a  partial order on $\K$ extending $\lea$. If $\clcc (\is, R)   = \aleph_0$ (i.e.\ $\bclcc (\is, R)$ contains all the regular cardinals), then $\bclc (\is, R) = \bclcwk (\is, R)$ and both are end segments of regular cardinals.
\end{cor}
\begin{proof}
  Directly from Fact \ref{locality-fact} (using that by definition $\bclcwk (\is, R)$ is always an end segment of regular cardinals).
\end{proof}
\begin{remark}
  The conclusion of Fact \ref{locality-fact} can be made into an equality assuming that $\is$ satisfies a weak transitivity property (see the statement for splitting and $R = \ltu$ in \cite[3.7]{ss-tame-jsl}). This is not needed in this paper.
\end{remark}

\section{Continuity of forking}\label{cont-sec}

In this section, we aim to study the locality cardinals and give conditions under which $\bclcc$ contains all regular cardinals. We work in an AEC with amalgamation and stability in a single cardinal $\mu$:

\begin{hypothesis}\label{cont-hyp} \
  \begin{enumerate}
  \item $\K$ is an AEC, $\mu \ge \LS (\K)$.
  \item $\K_\mu$ has amalgamation, joint embedding, and no maximal models in $\mu$. Moreover $\K$ is stable in $\mu$.
  \item $\is = (\K_\mu, \nf)$ is an independence relation.
  \end{enumerate}
\end{hypothesis}
\begin{remark}
  The results of this section generalize to AECs that may not have full amalgamation in $\mu$, but only satisfy the properties from \cite{shvi635}: density of amalgamation bases, existence of universal extensions, and limit models being amalgamation bases.
\end{remark}

We will usually assume that $\is$ has the weak uniqueness property:

\begin{defin}\label{weak-uq-def}
  $\is$ has \emph{weak uniqueness} if whenever $M_0 \lea M \lea N$ are all in $\K_\mu$ with $M$ universal over $M_0$, $p, q \in \gS (N)$ do not fork over $M_0$, and $p \rest M = q \rest M$, then $p = q$.
\end{defin}

The reader can think of $\is$ as non-$\mu$-splitting (Definition \ref{splitting-def}), where such a property holds \cite[I.4.12]{vandierennomax}. We state a more general version:

\begin{fact}[6.2 in \cite{tamenessone}]\label{splitting-weak-uq}
  $\issp{\mu} (\K_\mu)$ has weak uniqueness. More generally, let $M_0 \lea M \lea N$ all be in $\K_{\ge \mu}$ with $M_0 \in \K_\mu$. Assume that $M$ is universal over $M_0$ and $\K$ is $(\mu, \|N\|)$-tame (i.e.\ types over models of size $\|N\|$ are determined by their restrictions of size $\mu$).

  Let $p, q \in \gS (N)$. If $p, q$ both do not $\mu$-split over $M_0$ and $p \rest M = q \rest M$, then $p = q$.
\end{fact}

Interestingly, weak uniqueness implies a weak version of extension: 

\begin{lem}[Weak extension]\label{weak-ext}
  Let $M_0 \lea M \lea N$ all be in $\K_\mu$. Assume that $M$ is universal over $M_0$. Let $p \in \gS (M)$ and assume that $p$ does not fork over $M_0$.
  
  If $\is$ has weak uniqueness, then there exists $q \in \gS (N)$ extending $p$ such that $q$ does not fork over $M_0$.
\end{lem}
\begin{proof}
  We first prove the result when $M$ is $(\mu, \omega)$-limit over $M_0$. In this case we can write $M = M_\omega$, where $\seq{M_i : i \le \omega}$ is increasing continuous with $M_{i + 1}$ universal over $M_i$ for each $i < \omega$.
  
  Let $f: N \xrightarrow[M_1]{} M$. Let $q := f^{-1} (p)$. Then $q \in \gS (N)$ and by invariance $q$ does not fork over $M_0$. It remains to show that $q$ extends $p$. Let $q_M := q \rest M$. We want to see that $q_M = p$. By monotonicity, $q_M$ does not fork over $M_0$. Moreover, $q_M \rest M_1 = p \rest M_1$. By weak uniqueness, this implies that $q_M = p$, as desired.

  In the general case (when $M$ is only universal over $M_0$), let $M' \in \K_\mu$ be $(\mu, \omega)$-limit over $M_0$. By universality, we can assume that $M_0 \lea M' \lea M$. By the special case we have just proven, there exists $q \in \gS (N)$ extending $p \rest M'$ such that $q$ does not fork over $M_0$. By weak uniqueness, we must have that also $q \rest M = p$, i.e.\ $q$ extends $p$.
\end{proof}

We will derive continuity from weak uniqueness and the following locality property\footnote{In an earlier version, we derived continuity without any locality property but our argument contained a mistake.}, a weakening of locality from \cite[11.4]{baldwinbook09}:

\begin{defin}\label{weakly-local-def}
  Let $\chi$ be a regular cardinal. We say that an AEC $\K$ with a monster model is \emph{weakly $\chi$-local} if for any increasing continuous chain $\seq{M_i : i \le \chi}$ with $M_{i + 1}$ universal over $M_i$ for all $i < \chi$, if $p, q \in \gS (M_\chi)$ are such that $p \rest M_i = q \rest M_i$ for all $i < \chi$, then $p = q$. We say that $\K$ is \emph{weakly $(\ge \chi)$-local} if $\K$ is weakly $\chi'$-local for all regular $\chi' \ge \chi$.
\end{defin}

Note that any $(<\aleph_0)$-tame AEC (such as an elementary class, an AEC derived from homogeneous model theory, or even a universal class \cite{tameness-groups}, see also \cite[3.7]{ap-universal-apal}) is weakly $(\ge \aleph_0)$-local.

\begin{thm}\label{loc-implies-cont}
  Let $\chi < \mu^+$ be a regular cardinal. If $\K$ is weakly $\chi$-local and $\is$ has weak uniqueness, then $\chi \in \bclcc (\is, \ltu)$.
\end{thm}
\begin{proof}
  Let $\seq{M_i : i \le \chi}$ be increasing continuous in $\K_\mu$ with $M_{i + 1}$ universal over $M_i$ for all $i < \chi$. Let $p \in \gS (M_\chi)$ and assume that $p \rest M_i$ does not fork over $M_0$ for all $i < \chi$. Let $q \in \gS (M_\chi)$ be an extension of $p \rest M_1$ such that $q$ does not fork over $M_0$. This exists by weak extension (Lemma \ref{weak-ext}). By weak uniqueness, $p \rest M_i = q \rest M_i$ for all $i < \chi$. By weak $\chi$-locality, $p = q$, hence $p$ does not fork over $M_0$, as desired.
\end{proof}

In the rest of this paper, we will often look at $\mu$-splitting. The following notation will be convenient:

\begin{defin}\label{kappa-aec-def}
  Define $\bclc (\K_\mu, \ltu) := \bclc (\issp{\mu} (\K_\mu), \ltu)$. Similarly define the other variations in terms of $\bclcwk$ and $\bclcc$. Also define $\clc (\K_\mu, \ltu)$ and its variations.
\end{defin}

Note that any independence relation with weak uniqueness is extended by non-splitting. This is essentially observed in \cite[4.2]{bgkv-apal} but we give a full proof here for the convenience of the reader.

\begin{lem}\label{canon-lem}
  Assume that $\is$ has weak uniqueness.

  \begin{enumerate}
  \item Let $M_0 \lea M_1 \lea M$ all be in $\K_\mu$ such that $M_1$ is universal over $M_0$ and $M$ is universal over $M_1$. Let $p \in \gS (M)$. If $p$ does not fork over $M_0$, then $p$ does not $\mu$-split over $M_1$.
  \item $\bclc (\is, \ltu) \subseteq \bclc (\K_\mu, \ltu)$, and similarly for $\bclcwk$ and $\bclcc$.
  \end{enumerate}
\end{lem}
\begin{proof} \
  \begin{enumerate}
  \item Let $N_1, N_2 \in \K_\mu$ and $f : N_1 \cong_{M_1} N_2$ be such that $M_1 \lea N_\ell \lea M$ for $\ell = 1,2$. We want to see that $f (p \rest N_1) = p \rest N_2$. By monotonicity, $p \rest N_\ell$ does not fork over $M_0$ for $\ell = 1,2$. Consequently, $f (p \rest N_1)$ does not fork over $M_0$. Furthermore, $f (p \rest N_1) \rest M_1 = p \rest M_1 = (p \rest N_2) \rest M_1$. Applying weak uniqueness, we get that $f (p \rest N_1) = p \rest N_2$.
  \item Follows from the first part.
  \end{enumerate}
\end{proof}

\section{The stability spectrum of tame AECs}

For an AEC $\K$ with a monster model, we define the \emph{stability spectrum of $\K$}, $\Stab (\K)$ to be the class of cardinals $\mu \ge \LS (\K)$ such that $\K$ is stable in $\mu$. We would like to study it assuming tameness. From earlier work, the following is known about $\Stab (\K)$ in tame AECs:

\begin{fact}\label{stab-spec-facts}
  Let $\K$ be an $\LS (\K)$-tame AEC with a monster model.

  \begin{enumerate}
  \item\label{stab-spec-1} \cite[4.13]{sv-infinitary-stability-afml} If $\Stab (\K) \neq \emptyset$, then $\min (\Stab (\K)) < H_1$ (recall Notation \ref{basic-notation}).
  \item\label{stab-spec-2} \cite[6.4]{tamenessone}\footnote{Grossberg and VanDieren's proof shows that the assumption there that $\mu > H_1$ can be removed, see \cite[Theorem 12.10]{baldwinbook09}.} If $\mu \in \Stab (\K)$ and $\lambda = \lambda^\mu$, then $\lambda \in \Stab (\K)$.
  \item\label{stab-spec-3} \cite[1]{b-k-vd-spectrum} If $\mu \in \Stab (\K)$, then $\mu^+ \in \Stab (\K)$.
  \item\label{stab-spec-4} \cite[5.5]{ss-tame-jsl} If $\seq{\mu_i : i < \delta}$ is strictly increasing in $\Stab (\K)$ and $\cf{\delta} \in \bclc (\K_{\mu_0}, \ltu)$, then $\sup_{i < \delta} \mu_i \in \Stab (\K)$ (see Definition \ref{kappa-aec-def}).
  \end{enumerate}
\end{fact}

Let us say that $\K$ is \emph{stable} if $\Stab (\K) \neq \emptyset$. In this case, it is natural to give a name to the first stability cardinal:

\begin{defin}
  For $\K$ an AEC with a monster model, let $\lambda (\K) := \min (\Stab (\K))$ (if $\Stab (\K) = \emptyset$, let $\lambda (\K) := \infty$). 
\end{defin}

At the end of this section, we will define $\K$ to be \emph{superstable} if for all high-enough $\mu$, $\bclc (\K_{\mu}, \ltu)$ contains \emph{all} regular cardinals (Definition \ref{ss-def}). A byproduct of the work here is Corollary \ref{ss-cor}, showing that this is equivalent to being stable on a tail and hence justifying the ``superstable'' terminology. For now we more generally study the (possibly strictly) stable setup.

From Fact \ref{stab-spec-facts}, if $\K$ is an $\LS (\K)$-tame AEC with a monster model, then $\lambda (\K) < \infty$ implies that $\lambda (\K) < H_1$.

We will also rely on the following basic facts:

\begin{fact}[3.12 in \cite{bgkv-apal}]\label{splitting-card}
  Let $\K$ be an $\LS (\K)$-tame AEC with a monster model. For $M \lea N$, $p \in \gS (N)$, $\mu \in [\|M\|, \|N\|]$, $p$ $\mu$-splits over $M$ if and only if $p$ $\|M\|$-splits over $M$.
\end{fact}
\begin{fact}[3.3 in \cite{sh394}]\label{splitting-lc}
  Let $\K$ be an AEC with a monster model. Assume that $\K$ is stable in $\mu \ge \LS (\K)$. For any $M \in \K_{\ge \mu}$ and any $p \in \gS (M)$, there exists $M_0 \lea M$ with $M_0 \in \K_\mu$ such that $p$ does not $\mu$-split over $M_0$.  
\end{fact}

It is natural to look at the sequence $\seq{\bclc (\K_{\mu}, \ltu) : \mu \in \Stab (\K)}$ (recall Definition \ref{kappa-aec-def}). From \cite[\S4]{ss-tame-jsl}, we have that:

\begin{fact}
  Let $\K$ be an $\LS (\K)$-tame AEC with a monster model. If $\mu < \lambda$ are both in $\Stab (\K)$, then $\bclc (\K_\mu, \ltu) \subseteq \bclc (\K_\lambda, \ltu)$.
\end{fact}

Thus we define:

\begin{defin}\label{chi-def}
  For $\K$ an $\LS (\K)$-tame AEC with a monster model, let $\uchi (\K) := \bigcup_{\mu \in \Stab (\K)} \bclc (\K_{\mu (\K)}, \ltu)$. Let $\chi (\K)$ be the least regular cardinal $\chi$ such that $\chi' \in \uchi (\K)$ for any regular $\chi' \ge \chi$. Set $\uchi (\K) := \emptyset$ and $\chi (\K) := \infty$ if $\lambda (\K) = \infty$. 
\end{defin}
\begin{remark}\label{chi-rmk}
  By Fact \ref{splitting-lc}, $\chi (\K) \le \lambda (\K)^+$. Assuming continuity of splitting, we can prove that $\chi (\K) \le \lambda (\K)$ (see Theorem \ref{locality-card-bounds}).
\end{remark}
\begin{remark}
  If $\K$ comes from a first-order theory, then $\uchi (\K)$ is the set of regular cardinals greater than or equal to $\kappa_r (T)$, see Corollary \ref{fo-cor}.
\end{remark}

Fact \ref{splitting-lc} implies more generally that $[\lambda (\K), \infty) \cap \REG \subseteq \bclc (\K_\mu, \ltu)$ for any stability cardinal $\mu$. Thus we can let $\lambda' (\K)$ be the first place where the sequence of $\bclc (\K_\mu, \ltu)$ stabilizes. One can think of it as the first ``well-behaved'' stability cardinal.

\begin{defin}\label{lambdap-def}
  For $\K$ an $\LS (\K)$-tame AEC with a monster model, let $\lambda' (\K)$ be the least stability cardinal $\lambda$ such that $\bclc (\K_{\mu}, \ltu) \subseteq \bclc (\K_{\lambda}, \ltu)$ for all $\mu \in \Stab (\K)$. When $\lambda (\K) = \infty$, we set $\lambda' (\K) = \infty$.
\end{defin}

We do not know whether $\lambda' (\K) = \lambda (\K)$. In fact, while we know that $\lambda' (\K) < \infty$ if $\lambda (\K) < \infty$, we are unable to give any general bound at all on $\lambda' (\K)$. Assuming continuity of splitting, we can show that $\lambda' (\K) < \hanf{\lambda (\K)}$ (see Theorem \ref{locality-card-bounds}).

In this section, we prove what we can on $\uchi (\K)$ \emph{without} assuming continuity of splitting. Section \ref{cont-sec-thy} will prove more assuming continuity of splitting.

We will use the following fact, whose proof relies on the machinery of averages for tame AECs:

\begin{fact}[5.15,5.16 in \cite{bv-sat-afml}]\label{bv-sat}
  Let $\K$ be an $\LS (\K)$-tame AEC with a monster model.

  There exists a stability cardinal $\chi_0 < H_1$ such that for any $\mu > \mu_0 \ge \chi_0$, if:

  \begin{enumerate}
  \item\label{bv-sat-1} $\K$ is stable in unboundedly many cardinals below $\mu$.
  \item $\K$ is stable in $\mu_0$ and $\cf{\delta} \in \bclc (\K_{\mu_0}, \ltu)$
  \end{enumerate}

  then whenever $\seq{M_i : i < \delta}$ is an increasing chain of $\mu$-saturated models, we have that $\bigcup_{i < \delta} M_i$ is $\mu$-saturated.
\end{fact}

The following is the key result.

\begin{thm}\label{key-thm}
  Let $\K$ be an $\LS (\K)$-tame AEC with a monster model. Let $\chi_0 < H_1$ be as given by Fact \ref{bv-sat}. For any $\mu > \chi_0$, if $\K$ is stable in $\mu$ and in unboundedly many cardinals below $\mu$, then $\cf{\mu} \in \bclc (\K_\mu, \ltu)$.
\end{thm}

The proof of Theorem \ref{key-thm} will use the lemma below, which improves on \cite[3.17]{gv-superstability-jsl}.

\begin{lem}\label{key-lc-lem}
  Let $\K$ be an $\LS (\K)$-tame AEC with a monster model. Let $\delta$ be a limit ordinal and let $\seq{M_i : i \le \delta}$ be an increasing continuous sequence. If $M_\delta$ is $(\LS (\K) + \delta)^+$-saturated, then for any $p \in \gS (M_\delta)$, there exists $i < \delta$ such that $p$ does not $\|M_i\|$-split over $M_i$.
\end{lem}
\begin{proof}
  Assume for a contradiction that $p \in \gS (M_\delta)$ is such that $p$ $\|M_i\|$-splits over $M_i$ for every $i < \delta$. Then for every $i < \delta$ there exists $N_1^i, N_2^i, f_i$ such that $M_i \lea N_\ell^i \lea M$, $\ell = 1,2$, $f_i : N_1^i \cong_{M_i} N_2^i$, and $f_i (p \rest N_1^i) \neq p \rest N_2^i$. By tameness, there exists $M_1^i \lea N_1^i, M_2^i \lea N_2^i$ both in $\K_{\le \LS (\K)}$ such that $f_i[M_1^i] = M_2^i$ and $f_i (p \rest M_1^i) \neq p \rest M_2^i$.

  Let $N \lea M$ have size $\mu := \LS (\K) + \delta$ and be such that $M_\ell^i \lea N$ for $\ell = 1,2$ and $i < \delta$.

  Since $M_\delta$ is $\mu^+$-saturated, there exists $b \in |M_\delta|$ realizing $p \rest N$. Let $i < \delta$ be such that $b \in |M_i|$. By construction, we have that $f_i (p \rest M_1^i) \neq p \rest M_2^i$ but on the other hand $p \rest M_\ell^i = \gtp (b / M_\ell^i; M)$ and $f_i (p \rest M_1^i) = \gtp (b / M_2^i; M)$, since $f_i (b) = b$ (it fixes $M_i$). This is a contradiction.
\end{proof}

Before proving Theorem \ref{key-thm}, we show that Fact \ref{bv-sat} implies saturation of long-enough limit models:

\begin{thm}\label{sat-existence}
  Let $\K$ be an $\LS (\K)$-tame AEC with a monster model. Let $\chi_0 < H_1$ be as given by Fact \ref{bv-sat}. Let $\mu > \mu_0 \ge \chi_0$ be such that $\K$ is stable in $\mu_0, \mu$, and in unboundedly many cardinals below $\mu$.

  Then any $(\mu, \bclc (\K_{\mu_0}, \ltu) \cap \mu^+)$-limit model (see Section \ref{univ-sec}) is saturated. In particular, there is a saturated model of cardinality $\mu$.
\end{thm}
\begin{proof}
  Assume for simplicity that $\mu$ is limit (if $\mu$ is a successor cardinal, the proof is completely similar). Let $\gamma := \cf{\mu}$ and let $\seq{\mu_i : i < \gamma}$ be increasing cofinal in $\mu$ such that $\K$ is stable in $\mu_i$ for all $i < \gamma$. By Fact \ref{stab-spec-facts}(\ref{stab-spec-3}), $\K$ is stable in $\mu_i^+$ for all $i < \gamma$. Let $\delta \in \bclc (\K_{\mu_0}, \ltu) \cap \mu^+$. By Fact \ref{bv-sat}, for all $i < \gamma$, the union of a chain of $\mu_i$-saturated models of length $\delta$ is $\mu_i$-saturated. It follows that the $(\mu, \delta)$-limit model is saturated. Indeed, for each fixed $i < \gamma$, we can build an increasing continuous chain $\seq{M_j : j \le \delta}$ such that for all $j < \delta$, $M_j \in \K_\mu$, $M_{j + 1}$ is universal over $M_j$, and $M_{j + 1}$ is $\mu_i$-saturated. By what has just been observed, $M_{\delta}$ is $\mu_i$-saturated, and is a $(\mu, \delta)$-limit model. Now apply uniqueness of limit models of the same length.

  To see the ``in particular'' part, assume again that $\mu$ is limit (if $\mu$ is a successor, the $(\mu, \mu)$-limit model is saturated). Then without loss of generality, $\mu_0 > \lambda (\K)$, so by Fact \ref{splitting-lc}, $\lambda (\K)^+ \in \bclc (\K_{\mu_0}, \ltu)$. Thus the $(\mu, \lambda (\K)^+)$-limit model is saturated.
\end{proof}

\begin{proof}[Proof of Theorem \ref{key-thm}]
  Let $\mu > \chi_0$ be such that $\K$ is stable in $\mu$ and in unboundedly many cardinals below $\mu$. Let $\delta := \cf{\mu}$.

  By Theorem \ref{sat-existence}, there is a saturated model $M$ of cardinality $\mu$. Using that $\K$ is stable in unboundedly many cardinals below $\mu$, one can build an increasing continuous resolution $\seq{M_i : i \le \delta}$ such that $M_\delta = M$ and for $i < \delta$, $M_i \in \K_{<\mu}$, $M_{i + 1}$ is universal over $M_i$. By a back and forth argument, this shows that $M$ is $(\mu, \delta)$-limit. By Lemma \ref{key-lc-lem}, $\delta \in \bclc (\K_\mu, \ltu)$, as desired.
\end{proof}

\begin{cor}\label{key-cor}
  Let $\K$ be an $\LS (\K)$-tame AEC with a monster model. Let $\chi_0$ be as given by Fact \ref{bv-sat}. For any $\mu \ge \lambda' (\K) + \chi_0^+$ such that $\K$ is stable in unboundedly many cardinals below $\mu$, the following are equivalent:
  
  \begin{enumerate}
  \item\label{key-thm-1} $\K$ is stable in $\mu$.
  \item\label{key-thm-2} $\cf{\mu} \in \uchi (\K)$.
  \end{enumerate}
\end{cor}
\begin{proof}
  (\ref{key-thm-1}) implies (\ref{key-thm-2}) is by Theorem \ref{key-thm} and (\ref{key-thm-2}) implies (\ref{key-thm-1}) is by Facts \ref{stab-spec-facts}(\ref{stab-spec-3}),(\ref{stab-spec-4}).
\end{proof}

It is natural to ask whether Corollary \ref{key-cor} holds for \emph{arbitrary} high-enough $\mu$'s (i.e.\ without assuming stability in unboundedly many cardinals below $\mu$). At present, the answer we can give is sensitive to cardinal arithmetic: Fact \ref{stab-spec-facts} does not give us enough tools to answer in ZFC. There is however a large class of cardinals on which there is no cardinal arithmetic problems. This is already implicit in \cite[\S5]{ss-tame-jsl}.

\begin{defin}\label{closed-def}
  A cardinal $\mu$ is \emph{$\theta$-closed} if $\lambda^\theta < \mu$ for all $\lambda < \mu$. We say that $\mu$ is \emph{almost $\theta$-closed} if $\lambda^\theta \le \mu$ for all $\lambda < \mu$.
\end{defin}

\begin{lem}\label{closed-lem}
  Let $\K$ be an $\LS (\K)$-tame AEC with a monster model. If $\mu$ is almost $\lambda (\K)$-closed, then either $\mu = \mu^{\lambda (\K)}$ and $\K$ is stable in $\mu$, or $\K$ is stable in unboundedly many cardinals below $\mu$.
\end{lem}
\begin{proof}
  If $\mu^{\lambda (\K)} = \mu$, then $\K$ is stable in $\mu$ by Fact \ref{stab-spec-facts}(\ref{stab-spec-2}). Otherwise, $\mu$ is $\lambda (\K)$-closed. Thus for any $\mu_0 < \mu$, $\mu_1 := \mu_0^{\lambda (\K)}$ is such that $\mu_1 < \mu$ and $\mu_1^{\lambda (\K)} = \mu_1$, hence $\K$ is stable in $\mu_1$. Therefore $\K$ is stable in unboundedly many cardinals below $\mu$.
\end{proof}

We have arrived to the following application of Corollary \ref{key-cor}:

\begin{cor}[Eventual stability spectrum for closed cardinals]\label{key-cor-2}
  Let $\K$ be an $\LS (\K)$-tame AEC with a monster model. Let $\chi_0 < H_1$ be as given by Fact \ref{bv-sat}. Let $\mu$ be almost $\lambda (\K)$-closed with $\mu \ge \lambda' (\K) + \chi_0^+$. Then $\K$ is stable in $\mu$ if and only if $\cf{\mu} \in \uchi (\K)$.
\end{cor}
\begin{proof}
  If $\K$ is stable in unboundedly many cardinals below $\mu$, this is Corollary \ref{key-cor}. Otherwise by Lemma \ref{closed-lem}, $\K$ is stable in $\mu$ and $\mu^{\lambda (\K)} = \mu$. In particular, $\cf{\mu} > \lambda (\K)$, so by Fact \ref{splitting-lc}, $\cf{\mu} \in \uchi (\K)$.
\end{proof}
\begin{cor}\label{fo-cor}
  Let $\K$ be the class of models of a first-order stable theory $T$ ordered by $\preceq$. Then $\uchi (\K)$ is an end-segment and $\chi (\K) = \kappa_r (T)$ (the least regular cardinal greater than or equal to $\kappa (T)$).
\end{cor}
\begin{proof}
  Let $\chi$ be a regular cardinal and let $\mu := \beth_\chi (\lambda' (\K))$. If $\chi \ge \kappa (T)$, $\mu = \mu^{<\kappa (T)}$ so by the first-order theory $\K$ is stable in $\mu$. By Corollary \ref{key-cor-2}, $\chi \in \uchi (\K)$. Conversely, if $\chi \in \uchi (\K)$ then by Corollary \ref{key-cor-2}, $\K$ is stable in $\mu$, hence $\mu = \mu^{<\kappa (T)}$, so $\chi \ge \kappa (T)$.
\end{proof}

Note that the class of almost $\lambda (\K)$-closed cardinals forms a club, and on this class Corollary \ref{key-cor-2} gives a complete (eventual) characterization of stability. We do not know how to analyze the cardinals that are \emph{not} almost $\lambda (\K)$-closed in ZFC. Using hypotheses beyond ZFC, we can see that \emph{all} big-enough cardinals are almost $\lambda (\K)$-closed. In particular, Fact \ref{card-arith-fact} below shows that under SCH (and hence above a strongly compact) all high-enough cardinals are almost $\lambda (\K)$-closed. Thus (Corollary \ref{non-zfc-stab-spec}) in this setup we have a full characterization of the stability spectrum.

For ease of notation, we define the following function:

\begin{defin}\label{theta-def}
  For $\mu$ an infinite cardinal, $\theta (\mu)$ is the least cardinal $\theta$ such that any $\lambda \ge \theta$ is almost $\mu$-closed. When such a $\theta$ does not exist, we write $\theta (\mu) = \infty$.
\end{defin}

If $\lambda$ is a strong limit cardinal, then $2^{\lambda} = \lambda^{\cf{\lambda}}$ and so if $2^{\lambda} > \lambda^+$ we have that $\lambda^+$ is not almost $\cf{\lambda}$-closed. Foreman and Woodin \cite{gch-foreman-woodin} have shown that it is consistent with ZFC and a large cardinal axiom that $2^\lambda > \lambda^+$ for all infinite cardinals $\lambda$. Therefore it is possible that $\theta (\aleph_0) = \infty$ (and hence $\theta (\mu) = \infty$ for any infinite cardinal $\mu$). However, we have:

\begin{fact}\label{card-arith-fact}
  Let $\mu$ be an infinite cardinal.

  \begin{enumerate}
  \item If SCH holds, then $\theta (\mu) = 2^{\mu}$.
  \item If $\kappa > \mu$ is strongly compact, then $\theta (\mu) \le \kappa$.
  \end{enumerate}
\end{fact}
\begin{proof}
  The first fact follows from basic cardinal arithmetic (see \cite[5.22]{jechbook}), and the third follows from a result of Solovay (see \cite{solovay-sch} or \cite[20.8]{jechbook}).
\end{proof}

The following easy lemma will be used in the proof of Theorem \ref{sat-spec}:

\begin{lem}\label{limit-above-theta}
  Let $\K$ be an $\LS (\K)$-tame AEC with a monster model. If $\mu > \theta (\lambda (\K))$ and $\mu$ is limit, then $\K$ is stable in unboundedly many cardinals below $\mu$.
\end{lem}
\begin{proof}
  Let $\mu_0 \in [\theta (\lambda (\K)), \mu)$. As $\mu$ is limit, $\mu_0^+ < \mu$ and $\mu_0^+$ is almost $\lambda (\K)$-closed. In particular, $\mu_0^{\lambda (\K)} \le \mu_0^+$. By Fact \ref{stab-spec-facts}(\ref{stab-spec-2}), $\K$ is stable in either $\mu_0$ or $\mu_0^+$, as needed.
\end{proof}

\begin{cor}\label{non-zfc-stab-spec}
  Let $\K$ be an $\LS (\K)$-tame AEC with a monster model and let $\chi_0 < H_1$ be as given by Fact \ref{bv-sat}. For any $\mu \ge \lambda' (\K) + \chi_0^+ + \theta (\lambda (\K))$, $\K$ is stable in $\mu$ if and only if $\cf{\mu} \in \uchi (\K)$.
\end{cor}
\begin{proof}
  By Corollary \ref{key-cor-2} and the definition of $\theta (\lambda (\K))$.
\end{proof}

A particular case of Theorem \ref{key-thm} derives superstability from stability in a tail of cardinals. The following concept is studied already in \cite[6.3]{sh394}.

\begin{defin}\label{ss-def}
  An AEC $\K$ is \emph{$\mu$-superstable} if:

  \begin{enumerate}
  \item $\mu \ge \LS (\K)$.
  \item $\K_\mu$ is non-empty, has amalgamation, joint embedding, and no maximal models.
  \item $\K$ is stable in $\mu$.
  \item $\clc (\K_\mu, \ltu) = \aleph_0$ (i.e.\ $\bclc (\K_\mu, \ltu)$ consists of all the regular cardinals).
  \end{enumerate}
\end{defin}

This definition has been well-studied and has numerous consequences in tame AECs, such as the existence of a well-behaved independence notion (a good frame), the union of a chain of $\lambda$-saturated being $\lambda$-saturated, or the uniqueness of limit models (see for example \cite{gv-superstability-jsl} for a survey and history). Even though in tame AECs Definition \ref{ss-def} is (eventually) equivalent to all these consequences \cite{gv-superstability-jsl}, it was not known whether it followed from stability on a tail of cardinals (see \cite[1.7]{gv-superstability-jsl}). We show here that it does (note that this is a ZFC result).

\begin{cor}\label{ss-cor}
  Let $\K$ be an $\LS (\K)$-tame AEC with a monster model. The following are equivalent.

  \begin{enumerate}
  \item\label{ss-cor-1} $\chi (\K) = \aleph_0$ (i.e.\ $\uchi (\K)$ consists of all regular cardinals).
  \item\label{ss-cor-3} $\K$ is $\lambda' (\K)$-superstable.
  \item\label{ss-cor-2} $\K$ is stable on a tail of cardinals.
  \end{enumerate}
\end{cor}

The proof uses that $\mu$-superstability implies stability in every $\mu' \ge \mu$ (this is a straightforward induction using Fact \ref{stab-spec-facts}, see \cite[5.6]{ss-tame-jsl}). We state a slightly stronger version:

\begin{fact}[10.10 in \cite{indep-aec-apal}]\label{ss-upward}
  Let $\K$ be a $\mu$-tame AEC with amalgamation. If $\K$ is $\mu$-superstable, then $\K$ is $\mu'$-superstable for every $\mu' \ge \mu$.
\end{fact}

\begin{proof}[Proof of Corollary \ref{ss-cor}]
  If (\ref{ss-cor-1}) holds, then (\ref{ss-cor-3}) holds by definition of $\chi (\K)$. By Fact \ref{ss-upward}, this implies stability in every $\mu \ge \lambda' (\K)$, so (\ref{ss-cor-2}). Now if (\ref{ss-cor-2}) holds then by Corollary \ref{key-cor} we must have that $\chi (\K) = \aleph_0$, so (\ref{ss-cor-1}) holds.
\end{proof}

Corollary \ref{ss-cor} and the author's earlier work with Grossberg \cite{gv-superstability-jsl} justify the following definition for tame AECs:

\begin{defin}\label{ss-chi-def}
  Let $\K$ be an $\LS (\K)$-tame AEC with a monster model. We say that $\K$ is \emph{superstable} if $\chi (\K) = \aleph_0$.
\end{defin}

\section{The saturation spectrum}

Theorem \ref{sat-existence} shows that there is a saturated model in many stability cardinals. It is natural to ask whether this generalizes to all stability cardinals, and whether the converse is true, as in the first-order case. We show here that this holds assuming SCH, but prove several ZFC results along the way. Some of the proofs are inspired from the ones in homogeneous model theory (due to Shelah \cite{sh54}, see also the exposition in \cite{grle-homog}).

The following is standard and will be used without comments.

\begin{fact}\label{sat-inacc}
  Let $\K$ be an AEC with a monster model. If $\LS (\K) < \mu \le \lambda = \lambda^{<\mu}$, then $\K$ has a $\mu$-saturated model of cardinality $\lambda$.
\end{fact}

In particular, $\K$ has a saturated model in $\lambda$ if $\lambda = \lambda^{<\lambda}$.

We turn to studying what we can say about $\lambda$ when $\K$ has a saturated model in $\lambda$.

\begin{thm}\label{sat-stable-unbounded}
  Let $\K$ be an $\LS (\K)$-tame AEC with a monster model. Let $\LS (\K) < \lambda$. If $\K$ has a saturated model of cardinality $\lambda$ and $\K$ is stable in unboundedly many cardinals below $\lambda$, then $\K$ is stable in $\lambda$.
\end{thm}
\begin{proof}
  By Fact \ref{stab-spec-facts}(\ref{stab-spec-3}), we can assume without loss of generality that $\lambda$ is a limit cardinal. Let $\delta := \cf{\lambda}$. Pick $\seq{\lambda_i : i \le \delta}$ strictly increasing continuous such that $\lambda_\delta = \lambda$, $\lambda_0 \ge \LS (\K)$, and $i < \delta$ implies that $\K$ is stable in $\lambda_{i + 1}$. Let $M \in \K_\lambda$ be saturated and let $\seq{M_i : i \le \delta}$ be an increasing continuous resolution of $M$ such that for each $i < \delta$, $M_i \in \K_{\lambda_i}$ and $M_{i + 2}$ is universal over $M_{i + 1}$. 

  \underline{Claim}: For any $p \in \gS (M)$, there exists $i < \delta$ such that $p$ does not $\lambda$-split over $M_i$.

  \underline{Proof of Claim}: If $\delta > \lambda_1$, then the result follows from Facts \ref{splitting-card} and \ref{splitting-lc}. If $\delta \le \lambda_1$, then this is given by Lemma \ref{key-lc-lem}. $\dagger_{\text{Claim}}$

  Now assume for a contradiction that $\K$ is not stable in $\lambda$ and let $\seq{p_i : i < \lambda^+}$ be distinct members of $\gS (M)$ (the saturated model must witness instability because it is universal). By the claim, for each $i < \lambda^+$ there exists $j_i < \delta$ such that $p$ does not $\lambda$-split over $M_{j_i}$. By the pigeonhole principle, without loss of generality $j_i = j_0$ for each $i < \lambda^+$. Now $|\gS (M_{j_0})| \le |\gS (M_{j_0 + 2})| = \|M_{j_0 + 2}\| < \lambda$, so by the pigeonhole principle again, without loss of generality $p_i \rest M_{j_0 + 2} = p_j \rest M_{j_0 + 2}$ for all $i < j < \lambda^+$. By weak uniqueness of non-$\lambda$-splitting and tameness, this implies that $p_i = p_j$, a contradiction.
\end{proof}

We have not used the full strength of the assumption that $\K$ is stable in unboundedly many cardinals below $\lambda$. For example, the same argument as in Theorem \ref{sat-stable-unbounded} proves:

\begin{thm}\label{sat-stable-unbounded-2}
  Let $\K$ be an $\LS (\K)$-tame AEC with a monster model. Let $\LS (\K) < \lambda$ and let $M \in \K_\lambda$ be a saturated model. If $\lambda$ is singular and strong limit, then $\K$ is stable in $\lambda$.
\end{thm}

We can also prove in ZFC that existence of a saturated model at a cardinal $\lambda < \lambda^{<\lambda}$ implies that the class is stable. We first recall the definition of another locality cardinal:

\begin{defin}[4.4 in \cite{tamenessone}]\label{slc-def}
  For $\K$ a $\LS (\K)$-tame AEC with a monster model, define $\slc (\K)$ to be the least cardinal $\mu > \LS (\K)$ such that for any $M \in \K$ and any $p \in \gS (M)$, there exists $M_0 \in \K_{<\mu}$ with $M_0 \lea M$ such that $p$ does not $\|M_0\|$-split over $M_0$. Set $\slc (\K) = \infty$ if there is no such cardinal.
\end{defin}

We have that stability is equivalent to boundedness of $\slc (\K)$:

\begin{thm}\label{stab-mubar}
  Let $\K$ be an $\LS (\K)$-tame AEC with a monster model. The following are equivalent:

  \begin{enumerate}
  \item\label{mubar-1} $\K$ is stable.
  \item\label{mubar-2} $\slc (\K) < H_1$.
  \item\label{mubar-3} $\slc (\K) < \infty$.
  \end{enumerate}
\end{thm}
\begin{proof}
  (\ref{mubar-1}) implies (\ref{mubar-2}) is because by Fact \ref{splitting-lc}, $\slc (\K) \le \lambda (\K)^+$ and by Fact \ref{stab-spec-facts}(\ref{stab-spec-1}), $\lambda (\K)^+ < H_1$. (\ref{mubar-2}) implies (\ref{mubar-3}) is trivial. To see that (\ref{mubar-3}) implies (\ref{mubar-1}), let $\mu := \slc (\K)$. Pick any $\lambda_0 \ge \LS (\K)$ such that $\lambda_0 = \lambda_0^{<\mu}$ (e.g.\ $\lambda_0 = 2^{\mu}$), and pick any $\lambda > \lambda_0$ such that $\lambda^{\lambda_0} = \lambda$ (e.g.\ $\lambda = 2^{\lambda_0}$). We claim that $\K$ is stable in $\lambda$. Let $M \in \K_\lambda$, and extend it to $M' \in \K_\lambda$ that is $\mu$-saturated. It is enough to see that $|\gS (M')| = \lambda$, so without loss of generality $M = M'$. Suppose that $|\gS (M)| > \lambda$ and let $\seq{p_i : i < \lambda^+}$ be distinct members. By definition of $\mu$, for each $i < \lambda^+$ there exists $M_i \in \K_{<\mu}$ such that $M_i \lea M$ and $p$ does not $\|M_i\|$-split over $M_i$. Since $\lambda = \lambda^{<\mu}$, we can assume without loss of generality that $M_i = M_0$ for all $i < \lambda^+$. Further, $|\gS (M_0)| \le 2^{<\mu} \le \lambda_0^{<\mu} = \lambda_0$, so we can pick $M_0' \lea M$ with $M_0' \in \K_{\lambda_0}$ such that $M_0'$ is universal over $M_0$. As $\lambda = \lambda^{\lambda_0}$, we can assume without loss of generality that $p_i \rest M_0' = p_j \rest M_0'$. By tameness and weak uniqueness of non-splitting, we conclude that $p_i = p_j$, a contradiction.
\end{proof}

We will use that failure of local character of splitting allows us to build a tree of types, see the proof of \cite[4.6]{tamenessone}.

\begin{fact}\label{tree-building-fact}
  Let $\K$ be an $\LS (\K)$-tame AEC with a monster model. Let $\LS (\K) < \mu$, with $\mu$ a regular cardinal. If $\slc (\K) > \mu$, then there exists an increasing continuous tree $\seq{M_\eta : \eta \in \fct{\le\mu}{2}}$, and tree of types $\seq{p_\eta : \eta \in \fct{\le\mu}{2}}$, and sets $\seq{A_\eta : \eta \in \fct{\le\mu}{2}}$ such that for all $\eta \in \fct{<\mu}{2}$:

  \begin{enumerate}
  \item $M_{\eta} \in \K_{<\mu}$.
  \item $p_\eta \in \gS (M_\eta)$.
  \item $A_\eta \subseteq |M_{\eta \smallfrown 0}| \cap |M_{\eta \smallfrown 1}|$.
  \item $|A_\eta| \le \LS (\K)$.
  \item $p_{\eta \smallfrown 0} \rest A_\eta \neq p_{\eta \smallfrown 1} \rest A_\eta$.
  \end{enumerate}
\end{fact}

\begin{thm}\label{sat-stable}
  Let $\K$ be an $\LS (\K)$-tame AEC with a monster model. Let $\LS (\K) < \lambda < \lambda^{<\lambda}$. If $\K$ has a saturated model of cardinality $\lambda$, then $\slc (\K) \le \lambda^+$. In particular, $\K$ is stable.
\end{thm}
\begin{proof}
  The last sentence is Theorem \ref{stab-mubar}. Now suppose for a contradiction that $\slc (\K) > \lambda^+$.

  \underline{Claim}: $2^{<\lambda} = \lambda$.

  \underline{Proof of Claim}: Suppose not and let $\mu < \lambda$ be minimal such that $2^{\mu} > \lambda$. Then $\mu$ is regular so let $\seq{M_\eta : \eta \in \fct{\le\mu}{2}}$, $\seq{p_\eta : \eta \in \fct{\le\mu}{2}}$, and $\seq{A_\eta : \eta \in \fct{\le\mu}{2}}$ be as given by Fact \ref{tree-building-fact}. Since $2^{<\mu} \le \lambda$, we can use universality of $M$ to assume without loss of generality that $M_\eta \lea M$ for each $\eta \in \fct{< \mu}{2}$. By continuity of the tree, $M_\eta \lea M$ for each $\eta \in \fct{\mu}{2}$. Since $M$ is saturated, it realizes all types over $M_\eta$, for each $\eta \in \fct{\mu}{2}$. By construction of the tree, each of these types has a different realization so in particular, $2^{\mu} \le \lambda$, a contradiction. $\dagger_{\text{Claim}}$

  Now if there exists $\mu < \lambda$ such that $2^\mu = \lambda$, then $2^{\mu'} = \lambda$ for all $\mu' \in [\mu, \lambda)$, hence $\lambda = \lambda^{<\lambda}$, which we assumed was not true. Therefore $\lambda$ is strong limit. Since $\lambda < \lambda^{<\lambda}$, this implies that $\lambda$ is singular. By Theorem \ref{sat-stable-unbounded-2}, $\K$ is stable in $\lambda$. By Fact \ref{splitting-lc}, $\slc (\K) \le \lambda^+$, as desired.
\end{proof}

We have arrived to the following. Note that we need some set-theoretic hypotheses (e.g.\ assuming SCH, $\theta (H_1) = 2^{H_1}$, see Fact \ref{card-arith-fact}) to get that $\theta (H_1) < \infty$ otherwise the result holds vacuously.

\begin{cor}\label{sat-spec}
  Let $\K$ be an $\LS (\K)$-tame AEC with a monster model. Let $\chi_0 < H_1$ be as given by Fact \ref{bv-sat}. 
   Let $\lambda > \chi_0 + \theta (\lambda (\K))$ (recall Definition \ref{theta-def}). The following are equivalent:

  \begin{enumerate}
  \item\label{sat-spec-1} $\K$ has a saturated model of cardinality $\lambda$.
  \item\label{sat-spec-2} $\lambda = \lambda^{<\lambda}$ or $\K$ is stable in $\lambda$.
  \end{enumerate}
\end{cor}
\begin{proof}
  First assume (\ref{sat-spec-2}). If $\lambda = \lambda^{<\lambda}$, we get a saturated model of cardinality $\lambda$ using Fact \ref{sat-inacc}, so assume that $\K$ is stable in $\lambda$. If $\lambda$ is a successor, the $(\lambda, \lambda)$-limit model is saturated, so assume that $\lambda$ is limit. By Lemma \ref{limit-above-theta}, $\K$ is stable in unboundedly many cardinals below $\lambda$. By Theorem \ref{sat-existence}, $\K$ has a saturated model of cardinality $\lambda$.

  Now assume (\ref{sat-spec-1}) and $\lambda < \lambda^{<\lambda}$. By Theorem \ref{sat-stable}, $\K$ is stable. By Lemma \ref{closed-lem}, either $\K$ is stable in $\lambda$, or there are unboundedly many stability cardinals below $\lambda$. In the former case we are done and in the latter case, we can use Theorem \ref{sat-stable-unbounded}.
\end{proof}

When $\K$ is superstable (i.e.\ $\chi (\K) = \aleph_0$, see Definition \ref{ss-chi-def}), we obtain a characterization in ZFC.

\begin{cor}\label{ss-sat}
  Let $\K$ be an $\LS (\K)$-tame AEC with a monster model. The following are equivalent:

  \begin{enumerate}
  \item\label{ss-sat-1} $\K$ is superstable.
  \item\label{ss-sat-2} $\K$ has a saturated model of size $\lambda$ for every $\lambda \ge \lambda' (\K) + \LS (\K)^+$.
  \item\label{ss-sat-3} There exists $\mu$ such that $\K$ has a saturated model of size $\lambda$ for every $\lambda \ge \mu$.
  \end{enumerate}
\end{cor}
\begin{proof}
  (\ref{ss-sat-1}) implies (\ref{ss-sat-2}) is known (use Corollary \ref{ss-cor} to see that $\K$ is $\lambda' (\K)$-superstable, then apply \cite[6.10]{vv-symmetry-transfer-afml} together with \cite{vandieren-symmetry-apal}), and (\ref{ss-sat-2}) implies (\ref{ss-sat-3}) is trivial. Now assume (\ref{ss-sat-3}). By Theorem \ref{sat-stable}, $\K$ is stable. We prove by induction on $\lambda \ge \mu^{\lambda (\K)}$ that $\K$ is stable in $\lambda$. This implies superstability by Corollary \ref{ss-cor}.

  If $\lambda = \mu^{\lambda (\K)}$, then $\lambda^{\lambda (\K)} = \lambda$ so $\K$ is stable in $\lambda$ (see Fact \ref{stab-spec-facts}(\ref{stab-spec-2})). Now if $\lambda > \mu^{\lambda (\K)}$, then by the induction hypothesis $\K$ is stable in unboundedly many cardinals below $\lambda$, hence the result follows from Theorem \ref{sat-stable-unbounded}.
\end{proof}

\section{Characterizations of stability}

In \cite{gv-superstability-jsl}, Grossberg and the author characterize superstability in terms of the behavior of saturated, limit, and superlimit models. We show that stability can be characterized analogously. In fact, we are able to give a list of statements equivalent to ``$\chi \in \uchi (\K)$''.

\begin{remark}
  Another important characterization of superstability in \cite{gv-superstability-jsl} was solvability: roughly, the existence of an EM blueprint generating superlimit models. We do not know if there is a generalization of solvability to stability. Indeed it follows  from the proof of \cite[2.2.1]{shvi635}  that even an EM blueprint generating just universal (not superlimit) models would imply superstability (see also \cite{shvi-notes-apal}).
\end{remark}

We see the next definition as the ``stable'' version of a superlimit model. Very similar notions appear already in \cite{sh88}.

\begin{defin}
  Let $\K$ be an AEC. For $\chi$ a regular cardinal, $M \in \K_{\ge \chi}$ is \emph{$\chi$-superlimit} if:

  \begin{enumerate}
  \item $M$ has a proper extension.
  \item $M$ is universal in $\K_{\|M\|}$.
  \item For any increasing chain $\seq{M_i : i < \chi}$, if $i < \chi$ implies $M \cong M_i$, then $M \cong \bigcup_{i < \chi} M_i$.
  \end{enumerate}
\end{defin}

In \cite{gv-superstability-jsl}, it was shown that one of the statements below holds for all $\chi$ if and only if all of them hold for all $\chi$. The following characterization is a generalization to strictly stable AECs, where $\chi$ is fixed at the beginning.

\begin{thm}\label{charact-thm}
  Let $\K$ be a (not necessarily stable) $\LS (\K)$-tame AEC with a monster model. Let $\chi$ be a regular cardinal. The following are equivalent:

  \begin{enumerate}
    \setcounter{enumi}{-1}
  \item\label{charact-0} $\chi \in \uchi (\K)$.
  \item\label{charact-05} For unboundedly many $H_1$-closed stability cardinals $\mu$, $\cf{\mu} = \chi$.
  \item\label{charact-1} For unboundedly many cardinals $\mu$, there exists a saturated $(\mu, \chi)$-limit model.
  \item\label{charact-2} For unboundedly many $\mu$, the union of any increasing chain of $\mu$-saturated models of length $\chi$ is $\mu$-saturated.
  \item\label{charact-3} For unboundedly many stability cardinals $\mu$, there is a $\chi$-superlimit model of cardinality $\mu$.
  \item\label{charact-4} For unboundedly many $H_1$-closed cardinals $\mu$ with $\cf{\mu} = \chi$, there is a saturated model of cardinality $\mu$.
  \end{enumerate}
\end{thm}
\begin{proof}
  We first show that each of the conditions implies that $\K$ is stable. If (\ref{charact-0}) holds, then by definition of $\uchi (\K)$ we must have that $\K$ is stable. If (\ref{charact-1}) holds, then there exists in particular limit models and this implies stability. Also (\ref{charact-05}) and (\ref{charact-3}) imply stability by definition. If (\ref{charact-4}) holds, then we have stability by Theorem \ref{sat-stable}. Finally, assume that (\ref{charact-2}) holds. Build an increasing continuous chain of cardinals $\seq{\mu_i : i \le \chi}$ such that $\chi + \LS (\K) < \mu_0$, for each $i \le \chi$ any increasing chain of $\mu_i$-saturated models of length $\chi$ is $\mu_i$-saturated, and $2^{\mu_i} < \mu_{i + 1}$ for all $i < \chi$. Let $\mu := \mu_{\chi}$. Build an increasing chain $\seq{M_i : i < \chi}$ such that $M_{i + 1} \in \K_{2^{\mu_i}}$ and $M_{i + 1}$ is $\mu_i$-saturated. Now by construction $M := \bigcup_{i < \chi} M_i$ is in $\K_{\mu}$ and is saturated. Since $\cf{\mu} = \chi$, we have that $\mu < \mu^{\chi} \le \mu^{<\mu}$. By Theorem \ref{sat-stable}, $\K$ is stable. We have shown that we can assume without loss of generality that $\K$ is stable.

  We now show that (\ref{charact-2}) is equivalent to (\ref{charact-3}). Indeed, if we have a $\chi$-superlimit at a stability cardinal $\mu$, then it must be saturated and witnesses that the union of an increasing chain of $\mu$-saturated models of length $\chi$ is $\mu$-saturated. Conversely, we have shown in the first paragraph of this proof how to build a saturated model in a cardinal $\mu$ such that the union of an increasing chain of $\mu$-saturated models of length $\chi$ is $\mu$-saturated. Such a saturated model must be a $\chi$-superlimit.

  We also have that (\ref{charact-3}) implies (\ref{charact-1}), as it is easy to see that a $\chi$-superlimit model in a stability cardinal $\mu$ must be unique and also a $(\mu, \chi)$-limit model. Also, (\ref{charact-0}) implies (\ref{charact-2}) (Fact \ref{bv-sat}) and (\ref{charact-1}) implies (\ref{charact-0}) (Lemma \ref{key-lc-lem}). Therefore (\ref{charact-0}), (\ref{charact-1}), (\ref{charact-2}), (\ref{charact-3}) are all equivalent.

  Now, (\ref{charact-0}) implies (\ref{charact-4}) (Theorem \ref{sat-existence}). Also, (\ref{charact-4}) implies (\ref{charact-05}): Let $\mu$ be $H_1$-closed such that $\cf{\mu} = \chi$ and there is a saturated model of cardinality $\mu$. By Lemma \ref{closed-lem}, either $\K$ is stable in $\mu$ or stable in unboundedly many cardinals below $\mu$. In the latter case, Theorem \ref{sat-stable-unbounded} implies that $\K$ is stable in $\mu$. Thus $\K$ is stable in $\mu$, hence (\ref{charact-05}) holds.

  It remains to show that (\ref{charact-05}) implies (\ref{charact-0}). Let $\mu$ be an $H_1$-closed stability cardinal of cofinality $\chi$. By the proof of Lemma \ref{closed-lem}, $\K$ is stable in unboundedly many cardinals below $\mu$. By Theorem \ref{key-thm}, $\chi \in \uchi (\K)$, so (\ref{charact-0}) holds.
\end{proof}

\section{Indiscernibles and bounded equivalence relations}\label{indisc-seq}

We review here the main tools for the study of strong splitting in the next section: indiscernibles and bounded equivalence relations. All throughout, we assume:

\begin{hypothesis}
  $\K$ is an AEC with a monster model.
\end{hypothesis}

\begin{remark}
  By working more locally, the results and definitions of this section could be adapted to the amalgamation-less setup (see for example \cite[2.3]{categ-saturated-afml}).
\end{remark}

\begin{defin}[Indiscernibles, 4.1 in \cite{sh394}]
  Let $\alpha$ be a non-zero cardinal, $\theta$ be an infinite cardinal, and let $\seq{\ba_i : i < \theta}$ be a sequence of distinct elements each of length $\alpha$. Let $A$ be a set.

  \begin{enumerate}
  \item We say that \emph{$\seq{\ba_i : i < \theta}$ is indiscernible over $A$ in $N$} if for every $n < \omega$, every $i_0 < \ldots < i_{n - 1} < \theta$, $j_0 < \ldots < j_{n - 1} < \theta$, $\gtp (\ba_{i_0} \ldots \ba_{i_n} / A) = \gtp (\ba_{j_0} \ldots \ba_{j_n} / A)$. When $A = \emptyset$, we omit it and just say that $\seq{\ba_i : i < \theta}$ is indiscernible.
  \item We say that $\seq{\ba_i : i < \theta}$ is \emph{strictly indiscernible} if there exists an EM blueprint $\Phi$ (whose vocabulary is allowed to have arbitrary size) an automorphism $f$ of $\sea$ so that, letting $N' := \EM_{\tau (\K)} (\theta, \Phi)$:

    \begin{enumerate}
    \item For all $i < \theta$, $\bb_i := f (\ba_i) \in \fct{\alpha}{|N'|}$.
    \item If for $i < \theta$, $\bb_i = \seq{b_{i, j} : j < \alpha}$, then for all $j < \alpha$ there exists a unary $\tau (\Phi)$-function symbol $\rho_j$ such that for all $i < \theta$, $b_{i, j} = \rho_j^{N'} (i)$.
    \end{enumerate}
  \item Let $A$ be a set. We say that $\seq{\ba_i : i < \theta}$ is \emph{strictly indiscernible over $A$} if there exists an enumeration $\ba$ of $A$ such that $\seq{\ba_i \ba : i < \theta}$ is strictly indiscernible.
  \end{enumerate}
\end{defin}

Any strict indiscernible sequence extends to arbitrary lengths: this follows from a use of first-order compactness in the EM language. The converse is also true. This follows from the more general extraction theorem, essentially due to Morley:

\begin{fact}\label{indisc-extraction}
  Let $\BI := \seq{\ba_i : i < \theta}$ be distinct such that $\ell (\ba_i) = \alpha$ for all $i < \theta$. Let $A$ be a set. If $\theta \ge \hanf{\LS (\K) + |\alpha| + |A|}$, then there exists $\BJ := \seq{\bb_i : i < \omega}$ such that $\BJ$ is strictly indiscernible over $A$ and for any $n < \omega$ there exists $i_0 < \ldots < i_{n - 1} < \theta$ such that $\gtp (\bb_{0} \ldots \bb_{n - 1} / A) = \gtp (\ba_{i_0} \ldots \ba_{i_{n - 1}} / A)$.
\end{fact}

\begin{fact}\label{strict-indisc-charact}
  Let$\seq{\ba_i : i < \theta}$ be indiscernible over $A$, with $\ell (\ba_i) = \alpha$ for all $i < \theta$. The following are equivalent:

  \begin{enumerate}
  \item\label{strict-equiv-1} For any infinite cardinal $\lambda$, there exists $\seq{\bb_i : i < \lambda}$ that is indiscernible over $A$ and such that $\bb_i = \ba_i$ for all $i < \theta$.
  \item\label{strict-equiv-2} For all infinite $\lambda < \hanf{\theta + |A| + |\alpha| + \LS (\K)}$ (recall Notation \ref{basic-notation}), there exists $\seq{\bb_i : i < \lambda}$ as in (\ref{strict-equiv-1}).
  \item\label{strict-equiv-3} $\seq{\ba_i : i < \theta}$ is strictly indiscernible over $A$.
  \end{enumerate}
\end{fact}

We want to study bounded equivalence relations: they are the analog of Shelah's finite equivalence relations from the first-order setup but here the failure of compactness compels us to only ask for the number of classes to be bounded (i.e.\ a cardinal). The definition for homogeneous model theory appears in \cite[1.4]{hs-independence}. 

\begin{defin}
  Let $\alpha$ be a non-zero cardinal and let $A$ be a set. An \emph{$\alpha$-ary invariant equivalence relation on $A$} is an equivalence relation $E$ on $\fct{\alpha}{\sea}$ such that for any automorphism $f$ of $\sea$ fixing $A$, $\bb E \bc$ if and only if $f (\bb) E f (\bc)$.
\end{defin}
\begin{defin}
  Let $\alpha$ be a non-zero cardinal, $A$ be a set, and $E$ be an $\alpha$-ary invariant equivalence relation on $A$.

  \begin{enumerate}
    \item Let $c (E)$ be the number of equivalence classes of $E$.
    \item We say that $E$ is \emph{bounded} if $c (E) < \infty$ (i.e.\ it is a cardinal).
    \item Let $\SE^{\alpha} (A)$ be the set of $\alpha$-ary bounded invariant equivalence relations over $A$ ($S$ stands for strong).
  \end{enumerate}
\end{defin}

\begin{remark}\label{se-size}
  $$
  |\SE^{\alpha} (A)| \le |2^{\gS^{\alpha + \alpha} (A)}| \le 2^{2^{|A| + \LS (\K) + \alpha}}
  $$
\end{remark}

The next two results appear for homogeneous model theory in \cite[\S1]{hs-independence}. The main difference here is that strictly indiscernible and indiscernibles need not coincide.

\begin{lem}\label{n-equiv-classes-indisc}
  Let $E \in \SE^\alpha (A)$. Let $\BI$ be strictly indiscernible over $A$. For any $\ba, \bb \in \BI$, we have that $\ba E \bb$.
\end{lem}
\begin{proof}
  Suppose not, say $\neg (\ba  E \bb)$. Fix any infinite cardinal $\lambda \ge |\BJ|$. By Theorem \ref{strict-indisc-charact}, $\BI$ extends to a strictly indiscernible sequence $\BJ$ over $A$ of cardinality $\lambda$. Thus $c (E) \ge \lambda$. Since $\lambda$ was arbitrary, this contradicts the fact that $E$ was bounded.
\end{proof}

\begin{lem}\label{bounded-equiv-size}
  Let $A$ be a set and $\alpha$ be a non-zero cardinal. Let $E$ be an $\alpha$-ary invariant equivalence relation over $A$. The following are equivalent:

  \begin{enumerate}
  \item $E$ is bounded.
  \item $c (E) < \hanf{|A| + \alpha + \LS (\K)}$.
  \end{enumerate}
\end{lem}
\begin{proof}
  Let $\theta := \hanf{|A| + \alpha + \LS (\K)}$.
  If $c (E) < \theta$, $E$ is bounded. Conversely if $c (E) \ge \theta$ then we can list $\theta$ non-equivalent elements as $\BI := \seq{\ba_i : i < \theta}$. By Fact \ref{indisc-extraction}, there exists a strictly indiscernible sequence over $A$ $\seq{\bb_i : i < \omega}$ reflecting some of the structure of $\BI$. In particular, for $i < j < \omega$, $\neg (\bb_i E \bb_j)$. By Lemma \ref{n-equiv-classes-indisc}, $E$ cannot be bounded.
\end{proof}

The following equivalence relation will play an important role (see \cite[4.7]{hs-independence}):

\begin{defin}
  For all $A$ and $\alpha$, let $\Emin{A} := \bigcap \SE^\alpha (A)$.
\end{defin}

By Remark \ref{se-size} and a straightforward counting argument, we have that $\Emin{A} \in \SE^\alpha (A)$.

\section{Strong splitting}

We study the AEC analog of first-order strong splitting. It was introduced by Shelah in \cite[4.11]{sh394}. In the next section, the analog of first-order dividing will be studied. Shelah also introduced it \cite[4.8]{sh394} and showed how to connect it with strong splitting. After developing enough machinery, we will be able to connect Shelah's results on the locality cardinals for dividing \cite[5.5]{sh394} to the locality cardinals for splitting.

All throughout this section, we assume:

\begin{hypothesis}
  $\K$ is an AEC with a monster model.
\end{hypothesis}

\begin{defin}
  Let $\mu$ be an infinite cardinal, $A \subseteq B$, $p \in \gS (B)$. We say that $p$ \emph{$(<\mu)$-strongly splits over $A$} if there exists a strictly indiscernible sequence $\seq{\ba_i : i < \omega}$ over $A$ with $\ell (\ba_i) < \mu$ for all $i < \omega$ such that for any $b$ realizing $p$, $\gtp (b \ba_0 / A) \neq \gtp (b \ba_1 / A)$. We say that $p$ \emph{explicitly $(<\mu)$-strongly splits over $A$} if the above holds with $\ba_0 \ba_1 \in \fct{<\mu}{B}$. 

  $\mu$-strongly splits means $(\le \mu)$-strongly splits, which has the expected meaning. 
\end{defin}

\begin{remark}
  For $\mu < \mu'$, if $p$ [explicitly] $(<\mu)$-strongly splits over $A$, then $p$ [explicitly] $(<\mu')$-strongly splits over $A$. 
\end{remark}

\begin{lem}[Base monotonicity of strong splitting]\label{ss-base-monot}
  Let $A \subseteq B \subseteq C$ and let $p \in \gS (C)$. Let $\mu > |B \backslash A|$ be infinite. If $p$ $(<\mu)$-strongly splits over $B$, then $p$ $(<\mu)$-strongly splits over $A$.
\end{lem}
\begin{proof}
  Let $\seq{\ba_i : i < \omega}$ witness the strong splitting over $B$. Let $\bc$ be an enumeration of $B \backslash A$. The sequence $\seq{\ba_i \bc : i < \omega}$ is strictly indiscernible over $A$. Moreover, for any $b$ realizing $p$, $\gtp (b \bc \ba_0 / A) \neq \gtp (b \bc \ba_1 / A)$ if and only if $\gtp (b \ba_0 / A \bc) \neq \gtp (b \ba_1 / A \bc)$ if and only if $\gtp (b \ba_0 / B) \neq \gtp (b \ba_1 / B)$, which holds by the strong splitting assumption.
\end{proof}

Lemma \ref{ss-base-monot} motivates the following definition:

\begin{defin}\label{isssp-def}
  For $\lambda \ge \LS (\K)$), we let $\isssp{\lambda} (\K_\lambda)$ be the independence relation whose underlying class is $\K'$ and whose independence relation is non $\lambda$-strong-splitting.
\end{defin}

Next, we state a key characterization lemma for strong splitting in terms of bounded equivalence relations. This is used in the proof of the next result, a kind of uniqueness of the non-strong-splitting extension. It appears already for homogeneous model theory in \cite[1.12]{hs-independence}. 

\begin{defin}\label{sat-def}
  Let $N \in \K$, $A \subseteq |N|$, and $\mu$ be an infinite cardinal. We say that $N$ is \emph{$\mu$-saturated over $A$} if any type in $\gS^{<\mu} (A)$ is realized in $N$.
\end{defin}

\begin{lem}\label{weak-uq-lem-0}
  Let $N \in \K$ and let $A \subseteq |N|$. Assume that $N$ is $(\aleph_1 + \mu)$-saturated over $A$. Let $p := \gtp (b / N)$. The following are equivalent.

  \begin{enumerate}
  \item\label{weak-uq-0} $p$ does not \emph{explicitly} $(<\mu)$-strongly split over $A$.
  \item\label{weak-uq-1} $p$ does not $(<\mu)$-strongly split over $A$.
  \item\label{weak-uq-2} For all $\alpha < \mu$, all $\bc$, $\bd$ in $\fct{\alpha}{|N|}$, $\bc \Emin{A} \bd$ implies $\gtp (b \bc / A) = \gtp (b \bd / A)$.
  \end{enumerate}
\end{lem}
\begin{proof}
  If $p$ explicitly $(<\mu)$-strongly splits over $A$, then $p$ $(<\mu)$-strongly splits over $A$. Thus (\ref{weak-uq-1}) implies (\ref{weak-uq-0}).
  
  If $p$ $(<\mu)$-splits strongly over $A$, let $\BI = \seq{\ba_i : i < \omega}$ witness it, with $\ba_i \in \fct{\alpha}{|\sea|}$ for all $i < \omega$. By Lemma \ref{n-equiv-classes-indisc}, $\ba_0 \Emin{A} \ba_1$. However by the strong splitting assumption $\gtp (b \ba_0 / A) \neq \gtp (b \ba_1 / A)$. This proves (\ref{weak-uq-2}) implies (\ref{weak-uq-1}).

  It remains to show that (\ref{weak-uq-0}) implies (\ref{weak-uq-2}). Assume (\ref{weak-uq-0}). Assume $\bc$, $\bd$ are in $\fct{\alpha}{|N|}$ such that $\bc \Emin{A} \bd$. Define an equivalence relation $E$ on $\fct{\alpha}{|\sea|}$ as follows: $\bb_0 E \bb_1$ if and only if $\bb_0 = \bb_1$ or there exists $n < \omega$ and $\seq{\BI_i : i < n}$ strictly indiscernible over $A$ such that $\bb_0 \in I_0$, $\bb_1 \in I_{n - 1}$ and for all $i < n - 1$, $\BI_i \cap \BI_{i + 1} \neq \emptyset$ (this idea already appears in the proof of \cite[I.1.11(3)]{shelahfobook}; sometimes $n$ is called the \emph{Kim-Pillay distance} between $\bb_0$ and $\bb_1$). $E$ is an invariant equivalence relation over $A$. Moreover if $\seq{\ba_i : i < \theta}$ are in different equivalence classes and $\theta$ is sufficiently big, we can extract a strictly indiscernible sequence from it which will witness that all elements are actually in the same class. Therefore $E \in \SE^\alpha (A)$.

  Since $\bc \Emin{A} \bd$, we have that $\bc E \bd$ and without loss of generality $\bc \neq \bd$. Let $\seq{\BI_i : i < n}$ witness the finite Kim-Pillay distance. By saturation, we can assume without loss of generality that $\BI_i$ is in $|M|$ for all $i < n$. Now use the failure of explicit strong splitting to argue that $\gtp (b \bc / A) = \gtp (b \bd / A)$.
\end{proof}

\begin{lem}[Toward uniqueness of non strong splitting]\label{weak-uq-ss}
  Let $M \lea N$ and let $A \subseteq |M|$. Assume that $N$ is $(\aleph_1 +\mu)$-saturated over $A$ and for every $\alpha < \mu$, $\bc \in \fct{\alpha}{|N|}$, there is $\bd \in \fct{\alpha}{|M|}$ such that $\bd \Emin{A} \bc$. 

  Let $p, q \in \gS (N)$ not $(<\mu)$-strongly split over $A$. If $p \rest M = q \rest M$, then $p \rest B = q \rest B$ for every $B \subseteq |N|$ with $|B| < \mu$.
\end{lem}
\begin{proof}
  Say $p = \gtp (a / N)$, $q = \gtp (b / N)$. Let $\bc \in \fct{<\mu}{|N|}$. We want to see that $\gtp (a / \bc) = \gtp (b / \bc)$. We will show that $\gtp (a \bc / A) = \gtp (b \bc / A)$. Pick $\bd$ in $M$ such that $\bc \Emin{A} \bd$. Then by Lemma \ref{weak-uq-lem-0}, $\gtp (a \bc / A) = \gtp (a \bd / A)$. Since $p \rest M = q \rest M$, $\gtp (a \bd / A) = \gtp (b \bd / A)$. By Lemma \ref{weak-uq-lem-0} again, $\gtp (b \bd / A) = \gtp (b \bc / A)$. Combining these equalities, we get that $\gtp (a \bc / A) = \gtp (b \bc / A)$, as desired.
\end{proof}

\section{Dividing}
\begin{hypothesis}
  $\K$ is an AEC with a monster model.
\end{hypothesis}

The following notion generalizes first-order dividing and was introduced by Shelah \cite[4.8]{sh394}. 

\begin{defin}
  Let $A \subseteq B$, $p \in \gS (B)$. We say that $p$ \emph{divides over $A$} if there exists an infinite cardinal $\theta$ and a strictly indiscernible sequence $\seq{\bb_i : i < \theta}$ over $A$  as well as $\seq{f_i : i < \theta}$ automorphisms of $\sea$ fixing $A$ such that $\bb_0$ is an enumeration of $B$, $f_i (\bb_0) = \ba_i$ for all $i < \theta$, and $\seq{f_i (p) : i < \theta}$ is inconsistent.
\end{defin}

It is clear from the definition that dividing induces an independence relation:

\begin{defin}
  For $\lambda \ge \LS (\K)$, we let $\isdiv (\K_\lambda)$ be the independence relation whose underlying class is $\K'$ and whose independence relation is non-dividing.
\end{defin}

The following fact about dividing was proven by Shelah in \cite[5.5(2)]{sh394}:

\begin{fact}\label{dividing-lc}
  Let $\mu_1 \ge \mu_0 \ge \LS (\K)$. Let $\alpha < \mu_1^+$ be a regular cardinal. If $\K$ is stable in $\mu_1$ and $\mu_1^\alpha > \mu_1$, then $\alpha \in \bclcwk (\isdiv (\K_{\mu_0}))$ (recall Definition \ref{loc-def}).
\end{fact}

To see when strong splitting implies dividing, Shelah considered the following property:

\begin{defin} 
  $\K$ satisfies $(\ast)_{\mu, \theta, \sigma}$ if whenever $\seq{\ba_i : i < \delta}$ is a strictly indiscernible sequence with $\ell (\ba_i) < \mu$ for all $i < \delta$, then for any $\bb$ with $\ell (\bb) < \sigma$, there exists $u \subseteq \delta$ with $|u| < \theta$ such that for any $i,j \in \delta \backslash u$, $\gtp (\ba_i \bb / \emptyset) = \gtp (\ba_j \bb / \emptyset)$.
\end{defin}

\begin{fact}[4.12 in \cite{sh394}]\label{star-op}
  Let $\mu^\ast := \LS (\K) + \mu + \sigma$. If $\K$ does not have the $\mu^\ast$-order property (recall Definition \ref{op-def}), then $(\ast)_{\mu^+, \hanf{\mu^\ast}, \sigma^+}$ holds.
\end{fact}

\begin{lem}\label{dividing-ss}
  Let $A \subseteq B$. Let $p \in \gS (B)$. Assume that $(\ast)_{|B|^+, \theta, \sigma}$ holds for some infinite cardinals $\theta$ and $\sigma$.

  If $p$ explicitly $|B|$-strongly splits over $A$, then $p$ divides over $A$.
\end{lem}
\begin{proof}
  Let $\mu := |B|$. Let $\seq{\ba_i : i < \omega}$ witness the explicit strong splitting (so $\ell (\ba_i) = \mu$ for all $i < \omega$ and $\ba_0 \in \fct{\mu}{B}$). Increase the indiscernible to assume without loss of generality that $\ba_0$ enumerates $B$ and increase further to get $\seq{\ba_i : i < \theta^+}$. Pick $\seq{f_i : i < \theta^+}$ automorphisms of $\sea$ fixing $A$ such that $f_0$ is the identity and $f_i (\ba_0 \ba_1) = \ba_i \ba_{i + 1}$ for each $i < \theta^+$. We claim that $\seq{\ba_i : i < \theta^+}$, $\seq{f_i : i < \theta^+}$ witness the dividing over $A$.

  Indeed, suppose for a contradiction that $b$ realizes $f_i (p)$ for each $i < \theta^+$. In particular, $b$ realizes $f_0 (p) = p$. By $(\ast)_{\mu^+, \theta, \sigma}$, there exists $i < \theta^+$ such that $\gtp (b \ba_i / A) = \gtp (b \ba_{i + 1} / A)$. Applying $f_i^{-1}$ to this equation, we get that $\gtp (c \ba_0 / A) = \gtp (c \ba_1 / A)$, where $c := f_i^{-1} (b)$. But since $b$ realizes $f_i (p)$, $c$ realizes $p$. This contradicts the strong splitting assumption.
\end{proof}

We have arrived to the following result:

\begin{lem}\label{shelah-fact}
  Let $\mu_1 \ge \mu_0 \ge \LS (\K)$ be such that $\K$ is stable in both $\mu_0$ and $\mu_1$. Assume further that $\K$ does not have the $\mu_0$-order property.

  Let $\alpha < \mu_0^+$ be a regular cardinal. If $\mu_1^\alpha > \mu_1$, then:

  $$
  \alpha \in \bclcwk (\isssp{\mu_0} (\K_{\mu_0}), \ltu)
  $$
\end{lem}
\begin{proof}
  By Fact \ref{star-op}, $(\ast)_{\mu_0^+, \hanf{\mu_0}, \mu_0^+}$ holds.

  Now let $\seq{M_i : i < \alpha}$ be $\ltu$-increasing in $\K_\lambda$. Let $p \in \gS (\bigcup_{i < \alpha} M_i)$. By Fact \ref{dividing-lc}, there exists $i < \alpha$ such that $p \rest M_{i + 1}$ does not divide over $M_i$. By Lemma \ref{dividing-ss}, $p \rest M_{i + 1}$ does not explicitly $\lambda$-strongly split over $M_i$. By Lemma \ref{weak-uq-lem-0} (recall that $M_{i + 1}$ is universal over $M_i$), $p \rest M_{i + 1}$ does not $\lambda$-strongly split over $M_i$.
\end{proof}

Our aim in the next section will be to show that non-strong splitting has weak uniqueness. This will allow us to apply the results of Section \ref{cont-sec} and (assuming enough locality) replace $\bclcwk$ by $\bclc$.

\section{Strong splitting in stable tame AECs}\label{strong-split-sec}

\begin{hypothesis}\label{ss-tame-hyp}
  $\K$ is an $\LS (\K)$-tame AEC with a monster model.
\end{hypothesis}

Why do we assume tameness? Because we would like to exploit the uniqueness of strong splitting (Lemma \ref{weak-uq-ss}), but we want to be able to conclude $p = q$, and not just $p \rest B = q \rest B$ for every small $B$. This will allow us to use the tools of Section \ref{cont-sec}.

\begin{defin}\label{chi-ast-def}
  For $\mu \ge \LS (\K)$, let $\chi^\ast (\mu) \in [\mu^+, \hanf{\mu}$ be the least cardinal $\chi^\ast$ such that whenever $A$ has size at most $\mu$ and $\alpha < \mu^{+}$ then $c (\Emin{A}) < \chi^\ast$ (it exists by Lemma \ref{bounded-equiv-size}).
\end{defin}

The following is technically different from the $\mu$-forking defined in \cite[4.2]{ss-tame-jsl} (which uses $\mu$-splitting), but it is patterned similarly.

\begin{defin}
  For $p \in \gS (B)$, we say that \emph{$p$ does not $\mu$-fork over $(M_0, M)$} if:

  \begin{enumerate}
  \item $M_0 \lea M$, $|M| \subseteq B$.
  \item $M_0 \in \K_{\mu}$.
  \item $M$ is $\chi^\ast (\mu)$-saturated over $M_0$.
  \item $p$ does not $\mu$-strongly split over $M_0$.
  \end{enumerate}

  We say that $p$ does not $\mu$-fork over $M$ if there exists $M_0$ such that $p$ does not $\mu$-fork over $(M_0, M)$. 
\end{defin}

The basic properties are satisfied:

\begin{lem} \
  \begin{enumerate}
  \item (Invariance) For any automorphism $f$ of $\sea$, $p \in \gS (B)$ does not $\mu$-fork over $(M_0, M)$ if and only if $f(p)$ does not $\mu$-fork over $(f[M_0], f[M])$.
  \item (Monotonicity) Let $M_0 \lea M_0' \lea M \lea M'$, $|M'| \subseteq B$. Assume that $M_0, M_0' \in \K_\mu$ and $M$ is $\chi^\ast (\mu)$-saturated over $M_0'$

    Let $p \in \gS (B)$ be such that $p$ does not $\mu$-fork over $(M_0, M)$. Then:
    \begin{enumerate}
    \item $p$ does not $\mu$-fork over $(M_0', M)$.
    \item $p$ does not $\mu$-fork over $(M_0, M')$.
    \end{enumerate}
  \end{enumerate}
\end{lem}
\begin{proof}
  Invariance is straightforward. We prove monotonicity. Assume that $M_0, M_0', M, M', B, p$ are as in the statement. First we have to show that $p$ does not $\mu$-fork over $(M_0', M)$. We know that $p$ does not $\mu$-strongly split over $M_0$. Since $M_0' \in \K_\mu$, Lemma \ref{ss-base-monot} implies that $p$ does not $\mu$-strongly split over $M_0'$, as desired.

  Similarly, it follows directly from the definitions that $p$ does not $\mu$-fork over $(M_0, M')$.
\end{proof}

This justifies the following definition:

\begin{defin}
  For $\lambda \ge \LS (\K)$, we write $\isf{\mu} (\K_\lambda)$ for the independence relation with class $\K_\lambda$ and independence relation induced by non-$\mu$-forking.
\end{defin}

We now want to show that under certain conditions $\isf{\mu} (\K_\lambda)$ has weak uniqueness (see Definition \ref{weak-uq-def}). First, we show that when two types do not fork over the same sufficiently saturated model, then the ``witness'' $M_0$ to the $\mu$-forking can be taken to be the same.

\begin{lem}\label{no-witness-lem}
  Let $M$ be $\chi^\ast (\mu)$-saturated. Let $|M| \subseteq B$. Let $p, q \in \gS (B)$ and assume that both $p$ and $q$ do not $\mu$-fork over $M$. Then there exists $M_0$ such that both $p$ and $q$ do not $\mu$-fork over $(M_0, M)$.
\end{lem}
\begin{proof}
  Say $p$ does not fork over $(M_p, M)$ and $q$ does not fork over $(M_q, M)$. Pick $M_0 \lea M$ of size $\mu$ containing both $M_p$ and $M_q$. This works since $M$ is $\chi^\ast (\mu)$-saturated and $\chi^\ast (\mu) > \mu$.
\end{proof}

\begin{lem}\label{weak-uq-forking}
  Let $\mu \ge \LS (\K)$. Let $M \in \K_{\ge \mu}$ and let $B$ be a set with $|M| \subseteq B$. Let $p, q \in \gS (B)$ and assume that $p \rest M = q \rest M$.
  \begin{enumerate}
  \item (Uniqueness over $\chi^\ast$-saturated models) If $M$ is $\chi^\ast (\mu)$-saturated and $p, q$ do not $\mu$-fork over $M$, then $p = q$.
  \item (Weak uniqueness) Let $\lambda > \chi^\ast (\mu)$ be a stability cardinal. Let $M_0 \lea M$ be such that $M_0, M \in \K_{\lambda}$ and $M$ is universal over $M_0$. If $p, q$ do not $\mu$-fork over $M_0$, then $p = q$. In other words, $\isf{\mu} (\K_\lambda)$ has weak uniqueness.
  \end{enumerate}
\end{lem}
\begin{proof} \
  \begin{enumerate}
  \item By Lemma \ref{no-witness-lem}, we can pick $M_0$ such that both $p$ and $q$ do not $\mu$-fork over $(M_0, M)$. By Lemma \ref{weak-uq-ss}, $p = q$.
  \item Using stability, we can build $M' \in \K_\lambda$ that is $\chi^\ast (\mu)$-saturated with $M_0 \lea M'$. Without loss of generality (using universality of $M$ over $M_0$), $M' \lea M$. By base monotonicity, both $p$ and $q$ do not $\mu$-fork over $M'$. Since $p \rest M = q \rest M$, we also have that $p \rest M' = q \rest M'$. Now use the first part.
  \end{enumerate}
\end{proof}

The following theorem is the main result of this section, so we repeat its global hypotheses here for convenience.

\begin{thm}\label{stab-spec-forward}
  Let $\K$ be an $\LS (\K)$-tame AEC with a monster model.
  
  Let $\mu_0 \ge \LS (\K)$ be a stability cardinal. Let $\lambda > \chi^\ast (\mu_0)$ be another stability cardinal. For any $\mu_1 \ge \mu_0$, if $\K$ is stable in $\mu_1$ then $\mu_1^{<\clcwk (\K_\lambda, \ltu)} = \mu_1$ (recall Definition \ref{kappa-aec-def}).
\end{thm}

The proof will use the following fact (recall from Hypothesis \ref{ss-tame-hyp} that we are working inside the monster model of a tame AEC):

\begin{fact}[4.13 in \cite{sv-infinitary-stability-afml}]\label{op-equiv-fact}
  The following are equivalent:

  \begin{enumerate}
  \item $\K$ is stable.
  \item $\K$ does not have the $\LS (\K)$-order property.
  \end{enumerate}
\end{fact}
\begin{proof}[Proof of Theorem \ref{stab-spec-forward}]
  We prove that for any regular cardinal $\alpha < \lambda^+$, if $\mu_1^\alpha > \mu_1$ then $\alpha \in \bclcwk (\K_\lambda, \ltu)$. This suffices because the least cardinal $\alpha$ such that $\mu_1^\alpha > \mu_1$ is regular.

  Note that by definition $\bclcwk (\K_\lambda, \ltu)$ is an end segment of regular cardinals. Note also that by Lemma \ref{weak-uq-forking}, $\is := \isf{\mu_0} (\K_\lambda)$ has weak uniqueness and thus we can use the results from Section \ref{cont-sec} also on $\is$.

  By Fact \ref{splitting-card} and \ref{splitting-lc}, $\mu_0^+ \in \bclc (\K_\lambda, \ltu)$. Therefore we may assume that $\alpha < \mu_0^+$.

  By Fact \ref{op-equiv-fact}, $\K$ does not have the $\LS (\K)$-order property. By Lemma \ref{shelah-fact}, $\alpha \in \bclcwk (\isssp{\mu_0} (\K_{\mu_0}), \ltu)$. As in \cite[\S4]{ss-tame-jsl}, this implies that $\alpha \in \bclcwk (\isf{\mu_0} (\K_\lambda), \ltu)$. But by Lemma \ref{canon-lem}, this means that $\alpha \in \bclcwk (\K_\lambda, \ltu)$.
\end{proof}

\section{Stability theory assuming continuity of splitting}\label{cont-sec-thy}

In this section, we will assume that splitting has the weak continuity property studied in Section \ref{cont-sec}:

\begin{defin}\label{weak-cont-def}
  For $\K$ an AEC with a monster model, we say that \emph{splitting has weak continuity} if for any $\mu \in \Stab (\K)$, $\clcc (\K_\mu, \ltu) = \aleph_0$.
\end{defin}

Recall that Theorem \ref{loc-implies-cont} shows that splitting has weak continuity under certain locality hypotheses. In particular, this holds in any class from homogeneous model theory and any universal class.

Assuming continuity and tameness, we have that $\uchi (\K)$ is an end-segment of regular cardinals (see Corollary \ref{locality-cor}). Therefore $\chi (\K)$ is simply the minimal cardinal in $\uchi (\K)$. We have the following characterization of $\chi (\K)$:

\begin{thm}\label{concl-thm}
  Let $\K$ be a stable $\LS (\K)$-tame AEC with a monster model.

  If splitting has weak continuity, then $\chi (\K)$ is the maximal cardinal $\chi$ such that for any $\mu \ge \LS (\K)$, if $\K$ is stable in $\mu$ then $\mu = \mu^{<\chi}$.
\end{thm}
\begin{proof}
First, let $\mu \ge \LS (\K)$ be a stability cardinal. By Theorem \ref{stab-spec-forward} (recalling Corollary \ref{locality-cor}), $\mu^{<\chi (\K)} = \mu$.

Conversely, consider the cardinal $\mu := \beth_{\chi (\K)} (\lambda' (\K))$. By Fact \ref{stab-spec-facts}, $\K$ is stable in $\mu$. However $\cf{\mu} = \chi (\K)$ so $\mu^{\chi (\K)} > \mu$. In other words, there does not exist a cardinal $\chi > \chi (\K)$ such that $\mu^{<\chi} = \mu$.
\end{proof}

Still assuming continuity, we deduce an improved bound on $\chi (\K)$ (compared to Remark \ref{chi-rmk}) and an explicit bound on $\lambda' (\K)$:

\begin{thm}\label{locality-card-bounds}
  Let $\K$ be a stable $\LS (\K)$-tame AEC with a monster model and assume that splitting has weak continuity.

  \begin{enumerate}
  \item $\chi (\K) \le \lambda (\K) < H_1$.
  \item $\lambda (\K) \le \lambda' (\K) < \hanf{\lambda (\K)} < \beth_{H_1}$.
  \end{enumerate}
\end{thm}
\begin{proof} \
  \begin{enumerate}
  \item That $\lambda (\K)< H_1$ is Fact \ref{stab-spec-facts}. Now by Theorem \ref{concl-thm}, $\lambda (\K)^{<\chi (\K)} = \lambda (\K)$ and hence $\chi (\K) \le \lambda (\K)$.
  \item Let $\lambda'$ be the least stability cardinal above $\chi^\ast (\lambda (\K))$ (see Definition \ref{chi-ast-def}). We have that $\lambda' < \hanf{\lambda (\K)}$. We claim that $\lambda' (\K) \le \lambda'$. Indeed by Theorem \ref{stab-spec-forward}, for any stability cardinal $\mu$, we have that $\mu^{<\clc (\K_{\lambda'}, \ltu)} = \mu$. We know that $\chi (\K)$ is the maximal cardinal with that property, but on the other hand we have that $\chi (\K) \le \clc (\K_{\lambda'}, \ltu)$ by definition. We conclude that $\chi (\K) = \clc (\K_{\lambda'}, \ltu)$, as desired.
  \end{enumerate}
\end{proof}

Theorem \ref{locality-card-bounds} together with Corollary \ref{ss-cor} partially answers \cite[1.8]{gv-superstability-jsl}, which asked whether the least $\mu$ such that $\K$ is $\mu$-superstable must satisfy $\mu < H_1$. We know now that (assuming continuity of splitting) $\mu \le \lambda' (\K) < \beth_{H_1}$, so there is a Hanf number for superstability but whether it is $H_1$ (rather than $\beth_{H_1}$) remains open.

We also obtain an analog of Corollary \ref{non-zfc-stab-spec}: 

\begin{cor}\label{non-zfc-stab-spec-2}
  Let $\K$ be an $\LS (\K)$-tame AEC with a monster model and assume that splitting has weak continuity. For any $\mu \ge \lambda' (\K) + \theta (\lambda (\K))$, $\K$ is stable in $\mu$ if and only if $\mu = \mu^{<\chi (\K)}$.
\end{cor}
\begin{proof}
  The left to right direction follows from Theorem \ref{concl-thm} and the right to left direction is by Fact \ref{stab-spec-facts} and the definition of $\theta (\lambda (\K))$ (recalling that $\mu = \mu^{<\chi (\K)}$ implies that $\cf{\mu} \ge \chi (\K)$).
\end{proof}

We emphasize that for the right to left directions of Corollary \ref{non-zfc-stab-spec} to be nontrivial, we need $\theta (\lambda (\K)) < \infty$, which holds under various set-theoretic hypotheses by Fact \ref{card-arith-fact}. This is implicit in \cite[\S5]{ss-tame-jsl}. The left to right direction is new and does not need the boundedness of $\theta (\lambda (\K))$ (Theorem \ref{concl-thm}).

\subsection{On the uniqueness of limit models}

It was shown in \cite{limit-strictly-stable-v3} that continuity of splitting implies a nice local behavior of limit models in stable AECs:

\begin{fact}[Theorem 1 in \cite{limit-strictly-stable-v3}]\label{lim-uq-fact}
  Let $\K$ be an AEC and let $\mu \ge \LS (\K)$. Assume that $\K_\mu$ has amalgamation, joint embedding, no maximal models, and is stable in $\mu$. If:

  \begin{enumerate}
  \item $\delta \in \bclc (\K_{\mu}, \ltu) \cap \mu^+$.
  \item $\clcc (\K_{\mu}, \ltu) = \aleph_0$.
  \item $\K$ has $(\mu, \delta)$-symmetry.
  \end{enumerate}

  Then whenever $M_0, M_1, M_2 \in \K_\mu$ are such that both $M_1$ and $M_2$ are $(\mu, \ge \delta)$-limit over $M_0$ (recall Section \ref{univ-sec}), we have that $M_1 \cong_{M_0} M_2$.
\end{fact}

We will not need to use the definition of $(\mu, \delta)$-symmetry, only the following fact, which combines \cite[18]{limit-strictly-stable-v3} and the proof of \cite[2]{vandieren-symmetry-apal}.

\begin{fact}\label{sym-fact}
  Let $\K$ be an AEC and let $\mu \ge \LS (\K)$. Assume that $\K_\mu$ has amalgamation, joint embedding, no maximal models, and is stable in $\mu$. Let $\delta < \mu^+$ be a regular cardinal. If whenever $\seq{M_i : i < \delta}$ is an increasing chain of saturated models in $\K_{\mu^+}$ we have that $\bigcup_{i < \delta} M_i$ is saturated, then $\K$ has $(\mu, \delta)$-symmetry.
\end{fact}

We can conclude that in tame stable AECs with weak continuity of splitting, any two big-enough $(\ge \chi (\K))$-limits are isomorphic.

\begin{thm}\label{uq-lim-thm}
  Let $\K$ be an $\LS (\K)$-tame AEC with a monster model. Assume that splitting has weak continuity.

  Let $\chi_0 < H_1$ be as given by Fact \ref{bv-sat}. Then for any stability cardinal $\mu \ge \lambda' (\K) + \chi_0$ and any $M_0, M_1, M_2 \in \K_{\mu}$, if both $M_1$ and $M_2$ are $(\mu, \ge \chi (\K))$-limit over $M_0$, then $M_1 \cong_{M_0} M_2$.
\end{thm}
\begin{proof}
  By Fact \ref{bv-sat}, we have that the union of an increasing chain of saturated models in $\K_{\mu^+}$ of length $\chi (\K)$ is saturated. Therefore by Fact \ref{sym-fact}, $\K$ has $(\mu, \chi (\K))$-symmetry. Now apply Fact \ref{lim-uq-fact}.
\end{proof}

We deduce the following improvement on Theorem \ref{sat-existence} in case splitting has weak continuity:

\begin{cor}\label{sat-exist-cor}
  Let $\K$ be an $\LS (\K)$-tame AEC with a monster model. Assume that splitting has weak continuity.

  Let $\chi_0 < H_1$ be as given by Fact \ref{bv-sat}. For any stability cardinal $\mu \ge \lambda' (\K) + \chi_0$, there is a saturated model of cardinality $\mu$.
\end{cor}
\begin{proof}
  There is a $(\mu, \chi (\K))$-limit model of cardinality $\mu$, and it is saturated by Theorem \ref{uq-lim-thm}.
\end{proof}

In Fact \ref{bv-sat}, it is open whether Hypothesis (\ref{bv-sat-1}) can be removed. We aim to show that it can, assuming continuity of splitting and SCH. We first revisit an argument of VanDieren \cite{vandieren-chainsat-apal} to show that one can assume stability in $\lambda$ instead of stability in unboundedly many cardinals below $\lambda$.

\begin{lem}\label{chainsat-lem}
  Let $\K$ be an $\LS (\K)$-tame AEC with a monster model. Let $\mu > \LS (\K)$. Assume that $\K$ is stable in both $\LS (\K)$ and $\mu$. Let $\seq{M_i : i < \delta}$ be an increasing chain of $\mu$-saturated models. If $\cf{\delta} \in \bclc (\K_{\LS (\K)}, \ltu)$ and the $(\mu, \delta)$-limit model is saturated, then $\bigcup_{i < \delta} M_i$ is $\mu$-saturated.
\end{lem}

Let us say a little bit about the argument. VanDieren \cite{vandieren-chainsat-apal} shows that superstability in $\lambda$ and $\mu := \lambda^+$ combined with the uniqueness of limit models in $\lambda^+$ implies that unions of chains of $\lambda^+$-saturated models are $\lambda^+$-saturated. One can use VanDieren's argument to prove that superstability in unboundedly many cardinals below $\mu$ implies that unions of chains of $\mu$-saturated models are $\mu$-saturated, and this generalizes to the stable case too. However the case that interests us here is when $\K$ is stable in $\mu$ and not necessarily in unboundedly many cardinals below (the reader can think of $\mu$ as being the successor of a singular cardinal of low cofinality). This is where tameness enters the picture: by assuming stability e.g.\ in $\LS (\K)$ as well as $\LS (\K)$-tameness, we can transfer the locality of splitting upward and the main idea of VanDieren's argument carries through (note that continuity of splitting is not needed). Still several details have to be provided, so a full proof is given here.

\begin{proof}[Proof of Lemma \ref{chainsat-lem}]
  For $M_0 \lea M \lea N$, let us say that $p \in \gS (N)$ \emph{does not fork over $(M_0, M)$} if $M$ is $\|M_0\|^+$-saturated over $M_0$ (recall Definition \ref{sat-def}) and $M_0 \in \K_{\LS (\K)}$. Say that $p$ does not fork over $M$ if there exists $M_0$ so that it does not fork over $(M_0, M)$.

  Without loss of generality, $\delta = \cf{\delta} < \mu$. Let $M_\delta := \bigcup_{i < \delta} M_i$. Let $N \lea M_\delta$ with $N \in \K_{<\mu}$. Let $p \in \gS (N)$. We want to see that $p$ is realized in $M_\delta$. We may assume without loss of generality that $M_i \in \K_\mu$ for all $i \le \delta$. Let $q \in \gS (M_\delta)$ be an extension of $p$.

  Since $\delta \in \bclc (\K_{\LS (\K)}, \ltu)$, using \cite[\S4]{ss-tame-jsl} there exists $i < \delta$ such that $q$ does not fork over $M_i$. This means there exists $M_i^0 \lea M_i$ such that $M_i^0 \in \K_{\LS (\K)}$ and $q$ does not fork over $(M_i^0, M_i)$. Without loss of generality, $i = 0$. Let $\mu_0 := \LS (\K) + \delta$. Build $\seq{N_i : i \le \delta}$ increasing continuous in $\K_{\mu_0}$ such that $M_0^0 \lea N_0$, $N \lea N_\delta$, and for all $i \le \delta$, $N_i \lea M_i$. Without loss of generality, $N = N_\delta$.

  We build an increasing continuous directed system $\seq{M_i^\ast, f_{i, j} : i \le j < \delta}$ such that for all $i < \delta$:

  \begin{enumerate}
  \item $M_i^\ast \in \K_\mu$.
  \item $N_i \lea M_i^\ast \lea M_i$.
  \item $f_{i, i + 1}$ fixes $N_i$.
  \item $M_{i + 1}^\ast$ is universal over $M_i^\ast$.
  \end{enumerate}

  This is possible. Take $M_0^\ast := M_0$. At $i$ limit, take $M_i^{\ast \ast}$ to be the a direct limit of the system fixing $N_i$ and let $g: M_i^{\ast \ast} \xrightarrow[N_i]{} M_i$ (remember that $M_i$ is saturated). Let $M_i^\ast := g[M_i^{\ast \ast}]$, and define the $f_{j, i}$'s accordingly. At successors, proceed similarly and define the $f_{i, j}$'s in the natural way.

  This is enough. Let $(M_\delta^\ast, f_{i, \delta})_{i < \delta}$ be a direct limit of the system extending $N_\delta$ (note: we do \emph{not} know that $M_\delta^\ast \lea M_\delta$). We have that $M_\delta^\ast$ is a $(\mu, \delta)$-limit model, hence is saturated. Now find a saturated $\sea \in \K_\mu$ containing $M_\delta \cup M_\delta^\ast$ and such that for each $i < \delta$, there exists $f_{i, \delta}^\ast$ an automorphism of $\sea$ extending $f_{i, \delta}$ such that $f_{i, \delta}^\ast [N_\delta] \lea M_\delta^\ast$. This is possible since $M_\delta^\ast$ is universal over $M_i^\ast$ for each $i < \delta$. Let $N^\ast \lea M_\delta^\ast$ be such that $N^\ast \in \K_{\mu_0}$ and $|N_\delta| \cup \bigcup_{i < \delta} |f_{i, \delta}^\ast[N_\delta]| \subseteq |N^\ast|$.

  \underline{Claim}: For any saturated $\widehat{M} \in \K_\mu$ with $M_\delta \lea \widehat{M}$, there exists $\hat{q} \in \gS (\widehat{M})$ extending $q$ and not forking over $(M_0^0, N_0)$.

  \underline{Proof of Claim}: We know that $M_0$ is saturated. Thus there exists $f: M_0 \cong_{N_0} \widehat{M}$. Let $\hat{q} := f (q \rest M_0)$. We have that $\hat{q} \in \gS (\widehat{M})$ and $\hat{q}$ does not fork over $(M_0^0, N_0)$. Further, $\hat{q} \rest N_0 = q \rest N_0$. By uniqueness of nonforking (see \cite[5.3]{ss-tame-jsl}), $\hat{q} \rest M = q$. $\dagger_{\text{Claim}}$

  By the claim, there exists $\hat{q} \in \gS (\sea)$ extending $q$ and not forking over $(M_0^0, N_0)$. Because $M_\delta^\ast$ is $(\mu_0^+, \mu)$-limit, there exists $M^{\ast \ast} \in \K_\mu$ saturated such that $N^\ast \lea M^{\ast \ast} \lea M_{\delta}^\ast$ and $M_{\delta}^\ast$ is universal over $M^{\ast\ast}$.

  Since $M_\delta^\ast$ is universal over $M^{\ast \ast}$, there is $b^\ast \in M_\delta^\ast$ realizing $\hat{q} \rest M^{\ast \ast}$. Fix $i < \delta$ and $b \in M_i^\ast$ such that $f_{i, \delta} (b) = b^\ast$. We claim that $b$ realizes $p$ (this is enough since by construction $M_i^\ast \lea M_i \lea M_\delta$). We show a stronger statement: $b$ realizes $\hat{q} \rest M'$, where $M' := (f_{i, \delta}^\ast)^{-1}[M^{\ast \ast}]$. This is stronger because $N^\ast \lea M^{\ast \ast}$ so by definition of $N^\ast$, $N \lea (f_{i, \delta}^{\ast})^{-1}[N^\ast] \lea M'$. Work inside $\sea$. Since $\hat{q}$ does not fork over $(M_0^0, N_0)$, also $\hat{q} \rest M^{\ast \ast } = \gtp (b^\ast / M^{\ast \ast})$ does not fork over $(M_0^0, N_0)$. Therefore $\gtp (b / M')$ does not fork over $(M_0^0, N_0)$. Moreover, $\gtp (b / N_0) = \gtp (b^\ast / N_0) = \hat{q} \rest N_0$, since $f_{i, \delta}$ fixes $N_0$. By uniqueness, $\gtp (b / M') = \hat{q} \rest M'$. In other words, $b$ realizes $\hat{q} \rest M'$, as desired.
\end{proof}
\begin{remark}
  It is enough to assume that amalgamation and the other structural properties hold only in $\K_{[\LS (\K), \mu]}$.
\end{remark}

We have arrived to the second main result of this section. Note that the second case below is already known (Fact \ref{bv-sat}), but the others are new.

\begin{thm}\label{chainsat-thm}
  Let $\K$ be an $\LS (\K)$-tame AEC with a monster model. Assume that splitting has weak continuity.

  Let $\chi_0 < H_1$ be as given by Fact \ref{bv-sat}. Let $\lambda > \lambda' (\K) + \chi_0$ and let $\seq{M_i : i < \delta}$ be an increasing chain of $\lambda$-saturated models. If $\cf{\delta} \ge \chi (\K)$, then $\bigcup_{i < \delta} M_i$ is $\lambda$-saturated provided that at least one of the following conditions hold:

  \begin{enumerate}
  \item $\K$ is stable in $\lambda$.
  \item $\K$ is stable in unboundedly many cardinals below $\lambda$.
  \item $\lambda \ge \theta (\lambda (\K))$ (recall Definition \ref{theta-def}).
  \item SCH holds and $\lambda \ge 2^{\lambda (\K)}$.
  \end{enumerate}
\end{thm}
\begin{proof} \
  \begin{enumerate}
  \item We check that the hypotheses of Lemma \ref{chainsat-lem} hold, with $\K, \mu$ there standing for $\K_{\ge \lambda' (\K)}$, $\lambda$ here. By definition and assumption, $\K$ is stable in both $\lambda' (\K)$ and $\lambda$. Furthermore, $\cf{\delta} \in \bclc (\K_{\lambda' (\K)}, \ltu)$ by definition of $\lambda' (\K)$ and $\uchi (\K)$. Finally, any two $(\lambda, \ge \cf{\delta})$-limit models are isomorphic by Theorem \ref{uq-lim-thm}.
  \item If $\lambda$ is a successor, then $\K$ is also stable in $\lambda$ by Fact \ref{stab-spec-facts}(\ref{stab-spec-3}) so we can use the first part. If $\lambda$ is limit, then we can use the first part with each stability cardinal $\mu \in (\lambda' (\K) + \chi_0, \lambda)$ to see that the union of the chain is $\mu$-saturated. As $\lambda$ is limit, this implies that the union of the chain is $\lambda$-saturated.
  \item By definition of $\theta (\lambda (\K))$, $\lambda$ is almost $\lambda (\K)$-closed. By Lemma \ref{closed-lem}, either $\K$ is stable in $\lambda$ or $\K$ is stable in unboundedly many cardinals below $\lambda$, so the result follows from the previous parts.
  \item This is a special case of the previous part, see Fact \ref{card-arith-fact}.
  \end{enumerate}
\end{proof}

\section{Applications to existence and homogeneous model theory}\label{applications-sec}

We present here the following application of Lemma \ref{key-lc-lem}:

\begin{thm}\label{application-thm}
  Let $\K$ be an AEC and let $\lambda \ge \LS (\K)$. Assume that $\K$ has amalgamation in $\lambda$, no maximal models in $\lambda$, is stable in $\lambda$, and is categorical in $\lambda$. Let $\mu \le \lambda$ be a regular cardinal.

  If $\K$ is $(<\mu)$-tame, then $\K$ is $\lambda$-superstable (recall Definition \ref{ss-def}). In particular, it has a model of cardinality $\lambda^{++}$.
\end{thm}
\begin{remark}
  Here, $(<\mu)$-tameness is defined using Galois types over sets, see \cite[2.16]{sv-infinitary-stability-afml}.
\end{remark}

Theorem \ref{application-thm} can be seen as a partial answer to the question ``what stability-theoretic properties in $\lambda$ imply the existence of a model in $\lambda^{++}$?'' (this question is in turn motivated by the problem \cite[I.3.21]{shelahaecbook} of whether categoricity in $\lambda$ and $\lambda^+$ should imply existence of a model in $\lambda^{++}$). It is known that $\lambda$-superstability is enough \cite[8.9]{indep-aec-apal}. Theorem \ref{application-thm} shows that in fact $\lambda$-superstability is implied by categoricity, amalgamation, no maximal models, stability, and tameness.

Before proving Theorem \ref{application-thm}, we state a corollary to homogeneous model theory (see \cite{sh3} or the exposition in \cite{grle-homog}). The result is known in the first-order case \cite[VIII.0.3]{shelahfobook} but to the best of our knowledge, it is new in homogeneous model theory.

\begin{cor}
  Let $D$ be a homogeneous diagram in a first-order theory $T$. If $D$ is both stable and categorical in $|T|$, then $D$ is stable in all $\lambda \ge |T|$.
\end{cor}
\begin{proof}
  Let $\K_D$ be the class of $D$-models of $T$. It is easy to check that it is an $(<\aleph_0)$-tame AEC with a monster model. By Theorem \ref{application-thm}, $\K_D$ is $|T|$-superstable. Now apply Fact \ref{ss-upward}. 
\end{proof}

\begin{proof}[Proof of Theorem \ref{application-thm}]
  The ``in particular'' part is by the proof of \cite[8.9]{indep-aec-apal}, which shows that $\lambda$-superstability implies no maximal models in $\lambda^+$. We now prove that $\K$ is $\lambda$-superstable. For this it is enough to show that $\clc (\K_\lambda, \ltu) = \aleph_0$. So let $\delta < \lambda^+$ be a regular cardinal. We want to see that $\delta \in \bclc (\K_\lambda, \ltu)$. We consider two cases:

  \begin{itemize}
  \item \underline{Case 1: $\delta < \lambda$}. Let $\seq{M_i : i \le \delta}$ be $\ltu$-increasing continuous in $\K_\lambda$ and let $p \in \gS (M_\delta)$. Then by categoricity $M_\delta$ is $(\lambda, \delta^+ + \mu)$-limit, so the proof of Lemma \ref{key-lc-lem} directly gives that there exists $i < \delta$ such that $p$ does not $\lambda$-split over $M_i$.
  \item \underline{Case 2: $\delta = \lambda$}. Note first that this means $\lambda$ is regular. By \cite[I.3.3(2)]{sh394}, $\lambda \in \bclcwk (\K_\lambda, \ltu)$. By assumption, $\K$ is $(<\lambda)$-tame, thus it is weakly $\lambda$-local (recall Definition \ref{weakly-local-def}). By Theorem \ref{loc-implies-cont}, $\lambda \in \bclcc (\K_\lambda, \ltu)$. By Fact \ref{locality-fact}, $\delta = \lambda \in \bclc (\K_\lambda, \ltu)$, as desired.
  \end{itemize}
\end{proof}

\bibliographystyle{amsalpha}
\bibliography{stability-spectrum-aec}

\end{document}